\documentclass{compositio}
\usepackage{amssymb,amsmath,amsfonts,amsthm,epsfig,amscd}
\usepackage{stmaryrd}
\usepackage[all,cmtip,poly]{xy}

\usepackage{hyperref}

\newtheorem{theorem}[subsection]{Theorem}
\newtheorem{thm}[subsection]{Theorem}
\newtheorem{lemma}[subsection]{Lemma}
\newtheorem{lem}[subsection]{Lemma}

\newtheorem{cor}[subsection]{Corollary}

\newtheorem{prop}[subsection]{Proposition}
\newtheorem{proposition}[subsection]{Proposition}

\theoremstyle{definition}

\newtheorem{defn}[subsection]{Definition}

\newtheorem{definition}[subsection]{Definition}
\theoremstyle{remark}
\newtheorem{remark}[subsection]{Remark}

\newtheorem{example}[subsection]{Example}

\newtheorem{assumption}[subsection]{Assumption}

\def\numequation{\addtocounter{subsubsection}{1}\begin{equation}}
\def\nummultline{\addtocounter{subsubsubsection}{1}\begin{multline}}
\def\anumequation{\addtocounter{subsection}{1}\begin{equation}}

\renewcommand{\theequation}{\arabic{section}.\arabic{subsection}.\arabic{subsubsection}}

\newif\iffinalrun
\finalruntrue
\iffinalrun
\else
\fi

\iffinalrun
  \newcommand{\need}[1]{}
  \newcommand{\mar}[1]{}
\else
  \newcommand{\need}[1]{{\tiny *** #1}}
  \newcommand{\mar}[1]{\marginpar{\raggedright\tiny #1}}
\fi

\newcommand{\F}{\FF}

\newcommand{\Q}{\QQ}

\newcommand{\Z}{\ZZ}

\renewcommand{\AA}{{\mathbb A}}

\newcommand{\CC}{{\mathbb C}}

\newcommand{\FF}{{\mathbb F}}
\newcommand{\GG}{{\mathbb G}}
\newcommand{\HH}{{\mathbb H}}

\newcommand{\QQ}{{\mathbb Q}}
\newcommand{\RR}{{\mathbb R}}

\newcommand{\TT}{{\mathbb T}}

\newcommand{\ZZ}{{\mathbb Z}}

\newcommand{\bF}{\ensuremath{\mathbb{F}}}

\renewcommand{\bf}{\ensuremath{\mathbf{f}}}

\newcommand{\cO}{{\mathcal O}}

\newcommand{\rhobar}{\overline{\rho}}

\DeclareMathOperator{\Aut}{Aut}

\DeclareMathOperator{\Ext}{Ext}

\DeclareMathOperator{\Gal}{Gal}
\DeclareMathOperator{\GL}{GL}

\DeclareMathOperator{\Hom}{Hom}
\DeclareMathOperator{\im}{im}

\DeclareMathOperator{\ord}{ord}
\DeclareMathOperator{\PGL}{PGL}

\DeclareMathOperator{\rank}{rank}

\DeclareMathOperator{\Spec}{Spec}

\newcommand{\Frob}{\mathrm{Frob}}

\newcommand{\et}{\text{\'{e}t}}

\newcommand{\onto}{\twoheadrightarrow}

\DeclareMathOperator{\Sel}{Sel}
\DeclareMathOperator{\res}{res}

\DeclareMathOperator{\Iw}{Iw}

\newcommand{\legendre}[2]{\genfrac{(}{)}{}{}{#1}{#2}}

\begin{document}

\title{Level raising mod 2 and arbitrary 2-Selmer ranks}

\author[B. V. Le Hung]{Bao V. Le Hung}\email{lhvietbao@googlemail.com} 
\address{Department of Mathematics, University of Chicago,  5734 S. University Avenue, Chicago, Illinois 60637}
\author[Chao Li]{Chao Li}\email{chaoli@math.columbia.edu} 
\address{Department of Mathematics, Columbia University, 2990 Broadway, New York, NY 10027}


\begin{abstract}
We prove a level raising mod $\ell=2$ theorem for elliptic curves over $\mathbb{Q}$. It generalizes theorems of Ribet and Diamond--Taylor and also explains different sign phenomena compared to odd $\ell$. We use it to study the 2-Selmer groups of modular abelian varieties with common mod 2 Galois representation. As an application, we show that the 2-Selmer rank can be arbitrary in level raising families.  
\end{abstract}

\maketitle

\renewcommand{\labelenumi}{(\arabic{enumi})}

\section{Introduction}

Let $E/\mathbb{Q}$ be an elliptic curve of conductor $N$. The modularity theorem (\cite{Wiles1995,Taylor1995, Breuil2001}) associates to $E$ a weight 2 cusp newform $f$ of level $\Gamma_0(N)$. Let $\ell$ be a prime such that $E[\ell]$ is an absolutely irreducible $G_\mathbb{Q}=\Gal(\overline{\mathbb{Q}}/\mathbb{Q})$-representation. For primes $q\nmid N \ell$ satisfying the level raising condition $a_q\equiv \pm(q+1)$ mod $\ell$, Ribet's theorem \cite{Ribet1990} ensures the existence of a weight 2 cusp form $g$ of level $\Gamma_0(Nq)$ that is new at $q$ and  $g\equiv f\pmod{\ell}$.

When $\ell>2$, Diamond and Taylor (\cite{Diamond1994a},\cite{Diamond1994}) generalize Ribet's theorem and allow one to level raise at multiple primes $q_1,\ldots q_m$ simultaneously while keeping the form $g$ new at each $q_i$. At a prime $p$ where $g$ has conductor 1 (i.e. the primes $p||N$ and $p=q_i$), the local representation of $g$ is either the Steinberg representation or its unramified quadratic twist. The two cases are distinguished by the $U_p$-eigenvalue, which we call the \emph{sign of $g$} at $p$ (because it also dictates the sign of the local functional equation at $p$). At $p||N$, the sign of $g$ is the same as the sign of $f$ by the mod $\ell$ congruence since $\ell>2$. At $q_i\not\equiv -1\pmod{\ell}$, the sign of $g$ is uniquely determined by the mod $\ell$ congruence as well. At $q_i\equiv-1\pmod{\ell}$, both signs may occur as the sign of $g$.

When $\ell=2$, the signs cannot be detected from the mod $\ell$ congruence. In fact, it is not always possible to keep the same signs at all $p||N$ when level raising (see Example \ref{exa:signs} and \ref{exa:signs2}). Nevertheless, we are able to prove the following simultaneous level raising theorem for $\ell=2$, which allows one to prescribe any signs at $q_i$ and also keep the signs at all but one chosen $p||N$.

\begin{theorem}\label{thm:mainlevelraising}
  Let $E/\mathbb{Q}$ be an elliptic curve satisfying (1-4) of Assumption \ref{ass:main}.  Let $$f=\sum_{n\ge1}a_nq^n\in S_2(N)$$ be the newform associated to $E$. Let $q_1,\ldots,q_m$ be distinct level raising primes for $E$ (Definition \ref{def:levelraisingprime}). Given prescribed signs $\varepsilon_1,\ldots,\varepsilon_m\in\{\pm1\}$ and $\epsilon_p\in\{\pm 1\}$ for $p||N$, there exists a newform $$g=\sum_{n\ge1}b_nq^n\in S_2(N\cdot q_1\cdots q_m)$$ and a prime $\lambda$ of the Hecke field $F=\mathbb{Q}(\{b_n\}_{n\ge 1})$ above 2 such that $$b_p\equiv a_p\pmod{\lambda}, \quad p\nmid N\cdot q_1\cdots q_m; \quad\qquad b_{q_i}=\varepsilon_i,\quad i=1,\ldots, m,$$  and $$b_p=\epsilon_p,\quad \text{ for all but possibly one chosen }p|| N.$$
\end{theorem}

\begin{remark}
  See Theorem \ref{thm:levelraising} for a more general statement including sufficient conditions to when we can prescribe signs at all $p||N$.
\end{remark}

\begin{remark}
Given $f$, one may ask which signs can occur for congruent newforms $g$ of the same level $N$ (i.e., the case $m=0$). An argument of Ribet communicated to us shows that if $N=p$ is a prime, then there always exists a congruent newform $g$ of level $p$ with $U_p$-eigenvalue $+1$. This is best possible as the example $p=11$ shows: there is a unique newform of level $p=11$ and only $+1$ occurs as the $U_p$-eigenvalue. Though Theorem \ref{thm:mainlevelraising} does not treat this case, it follows from our method that if there is an odd number of primes $p||N$, then there always exists a congruent newform $g$ of level $N$ with $U_p$-eigenvalues $+1$ for all $p||N$ (see Remark \ref{rem:allpositivesigns}). This gives a different proof of Ribet's result in the case $N=p$.
\end{remark}

To such a level raised newform $g$, the Eichler--Shimura construction associates to it a modular abelian variety $A$ with real multiplication by $\mathcal{O}_F$.  We say that $A$ is \emph{obtained from $E$ via level raising} (mod 2). Then $E$ and $A$ are congruent mod 2, i.e., $$E[2] \otimes k\cong A[\lambda]$$ as $G_\mathbb{Q}$-representations, where $k=\mathcal{O}_F/\lambda$ is the residue field. In this way we can view both the 2-Selmer group $\Sel_2(E/\mathbb{Q}) \otimes k$ (extending scalars to $k$)   of $E$ and the $\lambda$-Selmer group $\Sel(A):=\Sel_\lambda(A/\mathbb{Q})$ of $A$ as $k$-subspaces of $H^1(\mathbb{Q}, E[2]\otimes k ) = H^1(\mathbb{Q}, A[\lambda])$ cut out by different local conditions. One may ask how the Selmer rank $\dim\Sel(A)$ is distributed when $A$ varies over all abelian varies obtained from $E$ via level raising. In particular, one may ask if $\dim\Sel(A)$ can take arbitrarily large or small values in this level raising family. We prove the following theorem, which gives an affirmative answer to the latter question.

\begin{theorem}\label{thm:mainarbitrary}
    Let $E/\mathbb{Q}$ be an elliptic curve satisfying Assumption \ref{ass:main}. Then for any given integer $n\ge0$, there exist infinitely many abelian varieties $A$ obtained from $E$ via level raising, such that $\dim\Sel(A)=n$. In particular, there exist infinitely many abelian varieties $A$ obtained from $E$ via level raising, such that $\rank A(\mathbb{Q})=0$.
\end{theorem}

\begin{remark}
 Mazur--Rubin \cite{Mazur2010} investigated 2-Selmer groups in quadratic twist families over arbitrary number fields in connection with Hilbert's tenth problem. In particular, they proved that if $\Gal(\mathbb{Q}(E[2])/\mathbb{Q})\cong S_3$ and $E$ has negative discriminant $\Delta$, then there exist infinitely many quadratic twists of $E/\mathbb{Q}$ of any given 2-Selmer rank. Theorem \ref{thm:mainarbitrary} is an analogue replacing quadratic twist families with level raising families.  In contrast to quadratic twisting, the level raising procedure never introduces places of additive reduction (at the cost of working with modular abelian varieties of higher dimension).
\end{remark}

\begin{remark}
  Our work is originally motivated by the recent work of W. Zhang \cite{Zhang2014}, who uses the level raising technique to prove the $\ell$-part of the Birch and Swinnerton-Dyer conjecture in the analytic rank one case when $\ell>3$. The strategy is to choose an auxiliary imaginary quadratic field $K$ (over which Heegner points exist) and prove that it is possible to lower the the $\ell$-Selmer rank over $K$ from one to zero via level raising (mod $\ell$). Then the Jochnowitz congruence of Bertolini--Darmon \cite{Bertolini1999} (relating the $\ell$-part of $L'(E/K,1)$ and $L(A/K,1)$) allows one to reduce the rank one case to the rank zero case, which is known thanks to the work of Skinner--Urban and Kato (see \cite[Theorem 2]{Skinner2014}). In contrast to Theorem \ref{thm:mainarbitrary}, it is not true that one can obtain arbitrary 2-Selmer rank \emph{over $K$}. In fact there is an obstruction for lowering the 2-Selmer rank over $K$ from one to zero as shown in \cite{Li2015}. Thus this strategy would not naively work for $\ell=2$ due to the aforementioned obstruction for rank lowering.
\end{remark}

We now outline the strategy of the proofs. The level raising problem with prescribed $U_p$-eigenvalues can be thought of as the problem of constructing a modular lift $\rho$ of the mod 2 representation $\overline{\rho}=\overline{\rho}_{E,2}$ with prescribed local types. For example, at a level raising prime $p$ we wish to force the local Galois representation $\rho|_{G_{\QQ_p}}$ to lie in the Steinberg or twisted Steinberg component (depending on the prescribed sign) of a local lifting ring. The usual technique to construct such lifts (e.g. as in \cite{Gee}, \cite{BLGGT}) is to write down a global deformation problem with prescribed local types, and then show that the deformation ring $R$ has modular points. This is achieved by showing that $R$ has positive Krull dimension, while at the same time being a finite $\ZZ_2$-algebra. The first fact is usually established by a Galois cohomology computation, whereas the second fact follows from a suitable modularity lifting theorem. In our situation, the Galois cohomology computation only shows $\dim R \geq 0$, while the image of $\overline{\rho}$ being dihedral causes trouble in applying modularity lifting theorems at $\ell=2$. When $\overline{\rho}$ is ordinary at 2, the modularity lifting theorem of P. Allen \cite{Allen2014} supplies the second ingredient (Theorem \ref{Prescribed type}). While the Krull dimension estimate fails, it is possible to salvage it by looking at a slightly different deformation problem, for which we prescribe the local types at all but one auxiliary prime, where we do not prescribe anything (Theorem \ref{Krull dim}). This allows us to construct (still in the ordinary case) the desired level raising form, except that it might be ramified at our auxiliary prime. However, with a well-chosen auxiliary prime, it turns out that the form thus constructed is either unramified or its quadratic twist is unramified (Corollary \ref{Level raising form at auxiliary prime}).  Twisting back allows us to get rid of this auxiliary prime at the cost of not prescribing the $U_p$-eigenvalue at one prime $p||N$. This establishes Theorem \ref{thm:mainlevelraising} in the ordinary case. This part of the argument generalizes well to totally real fields.

In the non-ordinary case, we do not have sufficiently strong modularity lifting theorems to make the above argument work. However, it turns out one can adapt the arguments of \cite{Diamond1994} in this case. In the definite case, the level raising result (Lemma \ref{Level raising definite}) is known to the experts (e.g., Kisin \cite{Kisin2009}). In the indefinite case, the crucial point is that while Fontaine--Laffaille theory breaks down at $\ell=2$, there is a version of it that works for unipotent objects (Lemma \ref{lem:Ihara}). This produces a level raising form (Prop. \ref{unrefined raising}), but with no control on the $U_p$-eigenvalues. One then shows that the existence of one such level raising form implies the existence of others, where we can change the $U_p$-eigenvalue. To make this work, we need to work with Shimura varieties at neat level, and thus we can only manipulate the signs at the cost of allowing ramification at an auxiliary prime. The same method in the previous paragraph will allow us to get rid of this auxiliary prime.

The proof of the main Theorem \ref{thm:mainarbitrary} consists of two parts: rank lowering (Theorem \ref{thm:ranklowering}) and rank raising (Theorem \ref{thm:rankraising}). In each case, we proceed by induction on the number of level raising primes.  When raising the level by one prime $q$, we can keep all the local conditions the same except at $q$ (Lemma \ref{lem:localconditions}). We then use a parity argument inspired by Gross--Parson \cite{Gross2012} to lower or raise the Selmer rank by one (Lemma \ref{lem:lowering}, \ref{lem:rankraising}). A Chebotarev density argument in fact shows that a positive density set of primes $q$ would work at each step (Prop. \ref{pro:ranklowering}, \ref{pro:rankraising}).

Implementing the parity argument encounters several complications for $\ell=2$.
\begin{enumerate}
\item The mod 2 Galois representation $\overline{\rho}=\overline{\rho}_{E,2}$ has small image ($\cong S_3$ under Assumption \ref{ass:main}) and $\Frob_q$ is either trivial or of order two for a level raising prime $q$. Thus it is not possible to choose $\Frob_q$ with distinct eigenvalues as in \cite{Gross2012}. Nevertheless, we can make use of the order two $\Frob_q$ to pin down the local condition at $q$ when the sign at $q$ is $+1$ (Lemma \ref{lem:localconditions} (3)).  Since the signs are not detected in the mod 2 congruence, it is crucial to have prescribed signs when raising the level, which is guaranteed by Theorem \ref{thm:mainlevelraising}.
\item  Since a finite group scheme over $\mathbb{Q}_2$ killed by 2 does not have a unique finite flat model over $\mathbb{Z}_2$ (unlike the case $\ell>2$), there is an extra uncertainty for the local condition at 2 even when $E$ has good reduction at 2. This uncertainty goes away when imposing Assumption \ref{ass:main} (4) by Lemma \ref{lem:localconditions} (4). This same assumption is also needed for proving the level raising Theorem \ref{thm:mainlevelraising} (see Remark  \ref{rem:2splits}).
\item In characteristic 2, it is crucial to work with not only the local Tate pairing but also a quadratic form giving rise to it (Remark \ref{rem:isotropic}). We utilize the quadratic form constructed by Zarhin \cite[\S 2]{Zarhin1974} using Mumford's Heisenberg group. Its properties were studied in O'Neil \cite{ONeil2002} and Poonen--Rains \cite{Poonen2012} and we provide an explicit formula for it in the proof of Lemma \ref{lem:isotropic}.
\end{enumerate}

The paper is organized as follows: Section \ref{Main definitions and examples} contains definitions and examples on level raising. Section \ref{sec:auxiliary-primes} discusses the auxiliary primes needed for level raising. Section \ref{sec:ordinarycase} and \ref{Simultaneous level raising: supersingular case} proves the level raising theorem. Section \ref{Preliminaries on local conditions} contains basic facts about Selmer groups. Section \ref{Rank lowering} and \ref{Rank raising} proves Theorem \ref{thm:mainarbitrary}.

\begin{acknowledgements}
We would like to thank
Brian Conrad,
Henri Darmon,
Benedict Gross,
Barry Mazur,
Bjorn Poonen,
Ken Ribet,
Richard Taylor and
Wei Zhang for helpful conversations or comments. We would also like to thank the referee for a careful reading and numerous suggestions. The examples in this article are computed using Sage (\cite{Stein2013}) and Magma (\cite{Bosma1997}). This material is also based upon work supported by the National Science Foundation under Grant No. 0932078000 while BLH was in residence at the Mathematical Sciences Research Institute in Berkeley, California, during the Fall 2014 semester.
\end{acknowledgements}

\section{Main definitions and examples}
\label{Main definitions and examples}
Let $E/\mathbb{Q}$ be an elliptic curve of conductor $N$. Let $\bar\rho=\bar\rho_{E,2}: G_{\mathbb{Q}}\rightarrow \Aut(E[2])\cong \GL_2(\mathbb{F}_2)$ be the Galois representation on the 2-torsion points. By the modularity theorem, $\bar\rho$ comes from a weight 2 cusp newform of level $N$. We make the following mild assumptions.

\begin{assumption}\label{ass:main}\quad
  \begin{enumerate}
\item $E$ has good or multiplicative reduction at 2 (i.e., $4\nmid N$). 
\item $\bar\rho$ is surjective and not induced from $\mathbb{Q}(i)$.
\item The Serre conductor $N(\bar\rho)$ is equal to the odd part of $N$. If $2\mid N$, $\bar\rho$ is ramified at 2.
\item If $2\nmid N$, $\bar\rho|_{G_{\mathbb{Q}_2}}$ is nontrivial;
\item $E$ has negative discriminant $\Delta$;
\end{enumerate}
\end{assumption}

\begin{remark}\label{rem:2torsionfield}
  The assumption (2) that $\bar \rho$ is surjective implies that the 2-torsion field $L=\mathbb{Q}(E[2])$ is a $\GL_2(\mathbb{F}_2)\cong S_3$-extension over $\mathbb{Q}$ and $\Gal(L/\mathbb{Q})\cong S_3$ acts on $E[2]$ via the 2-dimensional irreducible representation. The unique quadratic subextension of $L$ is $\mathbb{Q}(\sqrt{\Delta})$ (see \cite[p. 305]{Serre1972}). 
\end{remark}

\begin{remark}
  When $\bar\rho$ is induced from $\mathbb{Q}(i)$, a variant of Theorem \ref{thm:mainlevelraising} holds where we cannot control the ramification at one chosen $p|N$. Nevertheless, the proof of Theorem \ref{thm:mainarbitrary} still goes through in this case because the local condition at $p$ would be trivial.
\end{remark}

\begin{remark}\label{rem:componentgroup}
All the level raised forms will be automatically new at $p\mid N$ due to assumption (3).  This assumption is also equivalent to saying that the component group of the Neron model of $E$ at any $p\mid N$ has odd order (see \cite[Lemma 4]{Gross2012}).
\end{remark}

\begin{remark}\label{rem:2splits}
Notice that $\bar\rho|_{G_{\mathbb{Q}_2}}$ is trivial if and only if 2 splits in $L$, if and only if $E$ is ordinary at $2$ and $2$ splits in the quadratic subfield $\mathbb{Q}(\sqrt{\Delta})\subseteq L$. The assumption (4) is only needed for the proof of Lemma \ref{lem:localconditions} (4) and for fulfilling the last assumption of Theorem \ref{Prescribed type}. See Remark \ref{rem:localat2} and \ref{rem:trivialat2}.
\end{remark}

\begin{remark}\label{rem:negativedisc}
  The assumption (5) that $\Delta<0$ implies that the complex conjugation acts nontrivially on $E[2]$. The assumption (5) is needed for the proof of Theorem \ref{thm:mainarbitrary} (used in Lemma \ref{lem:localconditions} (2) and \ref{lem:raisingdensity}) but not for Theorem \ref{thm:mainlevelraising} (see Theorem \ref{thm:levelraising} and Remark \ref{rem:positivedisc}).
\end{remark}

Under these assumptions, $E[2]$ (as $G_\mathbb{Q}$-module) together with the knowledge of reduction type at a prime $q$  pins down the local condition defining $\Sel_2(E/\mathbb{Q})$ at $q$ (see Lemma \ref{lem:localconditions} for more precise statements). We would like to keep $E[2]$, but at a prime $q\nmid 2N$ of choice, to switch good reduction to multiplicative reduction and thus change the local condition at $q$. For this to happen, a necessary condition is that $\bar\rho(\Frob_q)=\left(\begin{smallmatrix}q & * \\ 0 & 1\end{smallmatrix}\right)\pmod{2}$ (up to conjugation). Namely, $\bar\rho(\Frob_q)=\left(\begin{smallmatrix}1 & 0 \\ 0 & 1\end{smallmatrix}\right)$ or $\left(\begin{smallmatrix}1 & 1 \\ 0 & 1\end{smallmatrix}\right)$ (order 1 or 2 in $S_3$).

\begin{definition}\label{def:levelraisingprime}
  We call $q\nmid 2N$ a \emph{level raising prime} for $E$ if $\Frob_q$ is of order 1 or 2 acting on $E[2]$. Notice that there are  lots of level raising primes: by the Chebotarev density theorem, they make up 2/3 of all primes. If we write $f=\sum_{n\ge1}a_nq^n\in S_2(N)$ (normalized so that $a_1=1$) to be the newform associated to the elliptic curve $E$. Then by definition $q\nmid 2N$ is a level raising prime for $E$ if and only if $a_q$ is even.
\end{definition}

The level raising theorem of Ribet ensures that this necessary condition is also sufficient.

\begin{theorem}[{\cite[Theorem 1]{Ribet1990}}]\label{thm:Ribet}
Assume $2\nmid N$, $\bar\rho$ is surjective and $N(\bar\rho)=N$. Let $q\nmid 2N$ be a level raising prime. Then $\bar\rho$ comes from a weight 2 newform of level $Nq$.
\end{theorem}

So whenever $q$ is a level raising prime, there exists a newform $g=\sum_{n\ge1}b_nq^n\in S_2(Nq)$ of level $Nq$ such that $$g\equiv f\pmod{2}.$$ More precisely, there exists a prime $\lambda\mid 2$ of the (totally real) Hecke field $F=\mathbb{Q}(\{b_n\}_{n\ge1})$ such that we have a congruence $b_p\equiv a_p\pmod{\lambda}$, for any $p\ne q$.

In the next two sections, we will prove the following theorem generalizing Theorem \ref{thm:Ribet}.

\begin{theorem}\label{thm:levelraising}
  Let $E/\mathbb{Q}$ be an elliptic curve satisfying (1--4) of Assumption \ref{ass:main}. Let $$f=\sum_{n\ge1}a_nq^n\in S_2(N)$$ be the newform associated to $E$. Let $q_1,\ldots,q_m$ be distinct level raising primes for $E$. Given any prescribed signs $\varepsilon_1,\ldots,\varepsilon_m\in\{\pm1\}$ and $\epsilon_p\in\{\pm 1\}$ for $p|| N$, there exists a newform $$g=\sum_{n\ge1}b_nq^n\in S_2(N\cdot q_1\cdots q_m)$$ and a prime $\lambda$ of the Hecke field $F=\mathbb{Q}(\{b_n\}_{n\ge 1})$ above 2 such that $$b_p\equiv a_p\pmod{\lambda}, \quad p\nmid N\cdot q_1\cdots q_m; \quad b_{q_i}=\varepsilon_i,\quad i=1,\ldots, m,$$  and $$b_p=\epsilon_p,\quad \text{ for all but possibly one chosen }p||N.$$ Moreover, if either of the following two assumptions holds,
  \begin{enumerate}
  \item  There exists $p|N$ such that $\ord_p(N)>1$ and $\ord_p(\Delta)$ is odd, or 
  \item $E$ has discriminant $\Delta>0$.
  \end{enumerate}
Then one can further require that $$b_p=\epsilon_p,\quad \text{ for all  }p||N.$$
\end{theorem}

\begin{remark}
  Our proof of this level raising theorem is divided into two parts according to whether $E$ is good ordinary or multiplicative at 2 (which we call the \emph{ordinary} case) or $E$ is good supersingular at 2 (which we call the \emph{supersingular} case). The proof in the ordinary case indeed only relies on the fact that $\bar\rho|_{G_{\mathbb{Q}_2}}$ is reducible.
\end{remark}

  This level raised newform $g$, via Eichler--Shimura construction, determines an abelian variety $A$ over $\mathbb{Q}$ up to isogeny, of dimension $[F:\mathbb{Q}]$, with real multiplication by $F$. We will choose an $A$ in this isogeny class so that $A$ admits an action by the maximal order $\mathcal{O}_F$. By Assumption \ref{ass:main} (2),  $A$ is unique up to a prime-to-$\lambda$ isogeny. By construction, for almost all primes $p$, $\Frob_p$ has same characteristic polynomials on $E[2]\otimes k$ and $A[\lambda]$. Hence by Chebotarev's density theorem and the Brauer--Nesbitt theorem we have $$E[2] \otimes k \cong A[\lambda]$$ as $G_\mathbb{Q}$-representations.

\begin{definition}\label{def:levelraising}
We say that $A$ is obtained from $E$ \emph{via level raising} at $q_1,\ldots q_m$ and that $A$ and $E$ are \emph{congruent mod 2}. We denote the \emph{sign of $A$} at $q_i$ by $\varepsilon_i(A)=\varepsilon_i$. 
\end{definition}

\begin{remark}
  We make the following convention: $E$ itself is understood as obtained from $E$ via level raising at $m=0$ primes. This is convenient for the induction argument later.
\end{remark}

\begin{example}
    Consider the elliptic curve $E=X_0(11):y^2+y=x^3-x^2-10x-20$ with Cremona's label $11a1$. We list the first few Hecke eigenvalues of the modular form $f$ associated to its isogeny class $11a$ in Table \ref{tab:level11}. We see that $q=7$ is a level raising prime (so are $q=13,17,19$). The space of newforms of level 77 has dimension 5, which corresponds to three isogeny classes of elliptic curves  $(77a, 77b, 77c)$ and one isogeny class of abelian surfaces $(77d)$. Among them $(77a, 77b)$ are congruent to $E$ mod 2: i.e., obtained from $E$ via level raising at 7. Their first few Hecke eigenvalues  are listed in Table \ref{tab:level11}.  Notice that both signs $\pm$ occur at 7 via level raising, but only the sign $-$ occurs at 11.
    \begin{table}[h]\caption{Level raising at 7}
    \centering
    \begin{tabular}[h]{*{9}{|c}|}
      &  2 & 3 & 5 & 7 & 11 & 13 & 17 & 19\\\hline
     $11a$ & $-2$ & $-1$& 1& $-2$& 1& 4& $-2$ & 0\\\hline
     $77a$ & 0 & $-3$ & $-1$ & $-1$ & $-1$ & $-4$ & 2 & $-6$\\\hline
     $77b$ & 0 & 1 & 3 & 1 & $-1$ & $-4$ & $-6$ & 2\\
    \end{tabular}
    \label{tab:level11}
  \end{table}
\end{example}

\begin{example}\label{exa:signs}
  We list all the possible signs in level raising families obtained from $E=X_0(11)$ in Table \ref{tab:11} at $(q_i)=(7), (13), (17), (7,13), (7,19)$. We have also included the dimension of the level raised abelian variety $A$ in the table. Notice that any prescribed signs at $q_i$ can occur as predicted by Theorem \ref{thm:levelraising}. But only one sign occurs at 11 for $(q_i)=(7), (13), (7,13), (7,19)$, 
    \begin{table}[h]\caption{$E=11a1$}\label{tab:11}
    \centering
\begin{tabular}[h]{|c|c|c|c|c|c||c|c|c|c|c|}
     & 11 & 7 & 13 & 17 & $\dim A$ &       & 11 & 7 & 13 &$\dim A$\\\hline
$11a$  & $+$  &   &    &   & 1 & $1001a$ & $+$  & $-$ & $-$ & 1 \\\hline
$77a$  & $-$  & $-$ &    & & 1   & $1001j$ & $+$  & $-$ & $+$ &5  \\\hline
$77b$  & $-$  & $+$ &    & & 1   & $1001k$ & $+$  & $+$ & $-$ & 5 \\\hline
$143a$ & $-$  &   & $-$  & & 1   & $1001n$ & $+$  & $+$ & $+$ & 11 \\\hline
$143c$ & $-$  &   & $+$  & & 6   &       & 11 & 7 & 19& \\\hline
$187a$ & $+$  &   &    & $-$ & 1 & $1463c$ & $+$  & $-$ & $+$ & 7 \\\hline
$187c$ & $+$  &   &    & $+$ & 2 & $1463e$ & $+$  & $+$ & $-$ & 9 \\\hline
$187d$ & $+$  &   &    & $-$ & 2 & $1463g$ & $+$  & $+$ & $+$ & 15 \\\hline
$187e$ & $-$  &   &    & $-$ & 3 & $1463i$ & $+$  & $-$ & $-$ & 16 \\\hline
$187f$ & $-$  &   &    & $+$ & 4 &       &    &   &   &
\end{tabular}
\end{table}
\end{example}

\begin{example}\label{exa:signs2}
 Table \ref{tab:35} illustrates possible signs obtained from $E=35a1: y^2 + y = x^3 + x^2 + 9x + 1$ via level raising at $q=19,23,31$. All possible 8 combinations of signs occur for $q=31$. For $q=19$, $23$, only the combination $(+,+)$ and $(-,-)$ occur as the signs at $(5,7)$. But all 4 combinations at $(5,q)$ or $(7,q)$ occur, as predicted by Theorem \ref{thm:levelraising}.
    \begin{table}[h]\caption{$E=35a1$}\label{tab:35}
    \centering
\begin{tabular}[h]{|c|c|c|c|c||c|c|c|c|c|}
  & 5 & 7 & 19 &    $\dim A$&                 &   5 & 7   & 31  &$\dim A$\\\hline
665a & $+$ & $+$ & $-$      & 1& 1085a & $+$ & $-$ & $+$ & 1 \\\hline
665b & $+$ & $+$ & $+$      & 1& 1085f & $+$ & $-$ & $-$ &1 \\\hline
665h & $-$ & $-$ & $+$      & 4& 1085g & $+$ & $-$ & $+$ &1 \\\hline
665i & $-$ & $-$ & $-$      & 6& 1085h & $+$ & $-$ & $-$ &1 \\\hline
  & 5 & 7 & 23 &              & 1085k & $-$ & $+$ & $+$  &3\\\hline
805c & $-$ & $-$ &     $-$  & 1& 1085l & $+$ & $-$ & $+$ &3 \\\hline
805d & $-$ & $-$ &     $+$  & 1& 1085m & $+$ & $-$ & $-$ &4 \\\hline
805g & $+$ & $+$ &     $-$  & 4& 1085n & $+$ & $+$ & $-$ &4 \\\hline
805m & $+$ & $+$ &     $+$  & 8& 1085o & $-$ & $-$ & $+$ &7 \\\hline
     &   &   &              & & 1085p & $-$ & $+$ & $-$  &7\\\hline
     &   &   &              & & 1085q & $-$ & $-$ & $-$  &8\\\hline
     &   &   &              & &1085r & $+$ & $+$ & $+$  &11\\
\end{tabular}
\end{table}     
\end{example}

\section{Auxiliary primes}\label{sec:auxiliary-primes}
In the next three sections, the elliptic curve $E$ is assumed to satisfy (1-4) of Assumption \ref{ass:main}. Recall that $\overline{\rho}:G_{\QQ}\to\GL_2(\FF_2)$ is the mod $2$ Galois representation of $E$ and $q_1,\ldots,q_m$ are distinct level raising primes for $E$.
\begin{defn} \label{Auxiliary prime definition} A prime $q_0\nmid Nq_1\cdots q_m$ is called an \emph{auxiliary} prime for $\overline{\rho}$ if 
\begin{enumerate}
\item $q_0\equiv3$ mod 4.
\item $\overline{\rho}(\Frob_{q_0})$ has order 3.
\item The Legendre symbol $\legendre{p}{q_0}=1$ for all $p|Nq_1\cdots q_m$ except one chosen prime $p= p_1$ such that $\ord_{p_1}(N)$ is odd.
\end{enumerate}
\end{defn}

\begin{lem}  \label{Existence of auxiliary primes} The set of auxiliary primes $q_0$ has positive density in the set of all primes.
\end{lem}
\begin{proof}  Observe that the first and last conditions are equivalent to demanding that $\Frob_{q_0}$ decomposes in a particular way in the extension $$M=\QQ(\sqrt{-1},\sqrt{p_2},\cdots \sqrt{p_s},\sqrt{q_1},\cdots \sqrt{q_m})/\QQ,$$ where $p_i$ are primes factors of $N$. On the other hand, the second condition is demanding that $\Frob_{q_0}$  has order 3 in $S_3=\Gal(\mathbb{Q}(E[2])/\mathbb{Q})$. Since $K=\QQ(\sqrt{\Delta})\subset \QQ(E[2])$ is the unique even degree subextension (Remark \ref{rem:2torsionfield}), it follows that $\QQ(E[2])/\mathbb{Q}$ and $M/\mathbb{Q}$ are linearly disjoint Galois extensions (by  Assumption \ref{ass:main} (2) that $K\ne \mathbb{Q}(\sqrt{-1})$). The Chebotarev density theorem thus implies the lemma. 
\end{proof}

\begin{remark}\label{rem:positivedisc}
  Note also that when $\Delta>0$, it is possible to get the third item also at $p=p_1$. This is because in this case $\QQ(E[2])\cap \QQ(\sqrt{-1},\sqrt{p_1},\cdots \sqrt{p_s},\sqrt{q_1},\cdots \sqrt{q_m})=\QQ(\sqrt{|\Delta|})=\QQ(\sqrt{\Delta})$, and the second and third requirement of $q_0$ give the same requirement on the splitting behavior of $\Frob_{q_0}$ in $\QQ(\sqrt{\Delta})$.
\end{remark}

The following lemma imposes a strong restriction on the lifts of $\overline{\rho}|_{G_{\QQ_{q_0}}}$ to characteristic $0$: 
\begin{lem} \label{Auxiliary prime} Let $\cO$ be a sufficiently large finite extension of $\ZZ_2$, with residue field $\FF$. Let $K$ be a finite extension of $\QQ_p$ with $p\neq 2$, whose residue field $k$ has order $q\equiv3$ mod 4. Let $\overline{r}:G_K\to \GL_2(\FF)$ be an unramified representation with $\det\overline{r}$ trivial and $\overline{r}(\Frob_K)$ has distinct eigenvalues in $\FF$. Suppose $r:G_K\to \GL_2(\cO)$ lifts $\overline{r}$ with cyclotomic determinant. Then $r|_{I_K} \otimes \eta$ is unramified, where $\eta:I_K\to \cO^\times$ is a quadratic character.
\end{lem}
\begin{proof} Let $P_K\subset I_K$ be the wild inertia group of $K$, and choose a tame generator $I_K/P_K=\langle \tau\rangle $. Let $\mathfrak{m}$ denote the maximal ideal of $\cO$. Let $\sigma$ be a choice of Frobenius of $K$. Because $\overline{r}$ is unramified, $r(P_K)=1$, and $r$ is determined by the two matrices $r(\sigma)$, $r(\tau)$ which are subject to the relations
\begin{align*} 
r(\sigma)r(\tau)r(\sigma)^{-1}&=r(\tau)^q,\\
\det r(\tau)&=1, \\
\det r(\sigma)&=q^{-1}.
\end{align*} 
 Without loss of generality, we may assume that $r(\sigma)=\begin{pmatrix} \alpha & 0\\ 0 & \beta \end{pmatrix}$, with $\bar{\alpha}\neq \bar{\beta}$. Writing $r(\tau)=1+\begin{pmatrix} a & b\\ c & d \end{pmatrix}$ with $a,b,c,d \in \mathfrak{m}$, we obtain
\[ 1+ \begin{pmatrix} a & \alpha\beta^{-1}b\\  \alpha^{-1}\beta c& d \end{pmatrix}=\left(1+ \begin{pmatrix} a & b\\  c& d \end{pmatrix}\right)^q. \] 
Expanding gives
\[ \begin{pmatrix} (1-q)a & (\alpha\beta^{-1}-q)b\\  (\alpha^{-1}\beta-q) c& (1-q)d \end{pmatrix}={{q}\choose{2}}  \begin{pmatrix} a & b\\  c& d \end{pmatrix}^2+\cdots +  \begin{pmatrix} a & b\\   c& d \end{pmatrix}^q.\]
Suppose we know $b,c \in \mathfrak{m}^k$ for some $k\geq 1$. Then comparing terms in the above equation we obtain $(\alpha\beta^{-1}-q)b \in \mathfrak{m}^{k+1}$, $(\alpha^{-1}\beta-q)c \in \mathfrak{m}^{k+1}$. Because $\alpha\beta^{-1}-q, \alpha^{-1}\beta-q$ are units (since $\bar\alpha\ne\bar\beta$), we have $b,c\in \mathfrak{m}^{k+1}$. Continuing inductively, we get $b,c \in \cap \mathfrak{m}^k=0$. Thus $r(\tau)$ must be diagonal, and must furthermore satisfy:
\begin{align*} &(1+a)^q=(1+a),\\ &(1+d)^q=(1+d),\\
&(1+a)(1+d)=1.
\end{align*}
Thus the image $r(I_K)$ is a subgroup of $\mu_{q-1}$. On the other hand because $\overline{r}(I_K)=1$, $r(I_K)$ is a pro-2 group, and since $q\equiv3$ mod 4, we must have $r(I_K)\subset \mu_2$.
\end{proof}
\begin{remark} \label{det trick} If we do not impose the condition that the determinant of the lift is the cyclotomic character (or just unramified), the same computation as above shows that the determinant of the lift restricted to $I_K$ is the unique non-trivial quadratic character.
\end{remark}
\begin{cor}\label{Level raising form at auxiliary prime} If $g$ is an eigenform with corresponding automorphic representation $\pi=\otimes \pi_p$ such that its mod $2$ representation $\overline{\rho}_{g}\cong \overline{\rho}$ and $q_0$ is an auxiliary prime, then either $\pi_{g}$ is unramified at $q_0$, or $\pi_q\otimes \chi_{q_0}$ is unramified, where $\chi_{q_0}$ is the unique quadratic character which is ramified at $q_0$ and unramified everywhere else.
\end{cor}

\begin{remark} The role of the auxiliary prime is to resolve the following tension: On the one hand, the space of automorphic forms that we need to investigate behaves well only when the level subgroup $U$ is ``sufficiently small'', on the other hand we want to construct automorphic forms with prescribed local behavior at \emph{all} primes. If the residual characteristic $\ell>2$ and a suitable largeness condition on the image of $\overline{\rho}$ holds, this problem can be resolved by allowing extra ramification at an auxiliary prime with the property that any automorphic form congruent to $\overline{\rho}$ will automatically be unramified at the auxiliary prime (This is what is done in \cite{Diamond1994}, for example). In the situation we are interested in, it is not possible to find auxiliary primes that achieves this, however Corollary \ref{Level raising form at auxiliary prime} shows that we can assure that automorphic forms lifting $\overline{\rho}$ will have at most quadratic ramification at the auxiliary primes. This turned out to be sufficient for our purposes, by making a quadratic twist to get rid of the extra ramification.  
\end{remark}

\section{Simultaneous level raising: ordinary case}
\label{sec:ordinarycase}

We fix a finite extension $E$ of $\mathbb{Q}_2$ which is sufficiently large, with ring of integers $\cO$ and residue field $\bF$. Let $\pi$ denote a uniformizer. 
Let $\mathfrak{Ar}_\cO$ be the category of Artinian local $\cO$ algebras with residue field identified with $\FF$ via the $\cO$-algebra structure. Let $F$ be a totally real field. We fix a finite set of places $S$ of $F$, and a subset $\Sigma\subset S$ which contains all places $v|p$ and $v|\infty$. Let $\overline{\rho}:G_{F,S}\to \GL_2(\bF)$ be an absolutely irreducible representation, where $G_{F,S}$ is the Galois group of the maximal extension of $F$ unramified outside the finite places in $S$. Denote by $V_\bF=\bF^2$ the $G_{F,S}$-module induced by $\overline{\rho}$ and $\beta_\bF$ the standard basis of $V_\bF$. Let $\psi: G_{F,S} \to \cO^\times$ be a character lifting $\det \overline{\rho}$. If $v|2$, denote by $\Lambda(G_{F_v})$ the completed group algebra $\cO[[G_{F_v}^{\mathrm{ab}}(2)]]$ of the maximal pro-2 quotient of $G_{F_v}^{\mathrm{ab}}$. It is a complete local Noetherian commutative ring with residue field $\FF$, and we let $\mathfrak{Ar}_{\Lambda(G_{F_v})}$ the category of local Artinian $\Lambda(G_{F_v})$-algebras with residue field $\FF$. Note that $\Lambda(G_{F_v})$ carries the universal character that is trivial $\bmod{\pi}$. Similarly, let $\Lambda'(G_{F_v})$ be the completed group algebra  $\cO[[I_{{F_v}^{\mathrm{ab}}}(2)]]$. It is a subalgebra of $\Lambda(G_{F_v})$, and the restriction of the aforementioned universal character to $I_{F_v}$ takes values in this subalgebra.

For each place $v$ of $F$, let $D_v$ be the functor on $\mathfrak{Ar}_\cO$ that assigns to $(A,m_A) \in \mathfrak{Ar}_\cO$ the set of isomorphism classes of tuples $(V_A,\iota_A,\beta)$ where $V_A$ is a finite free $A$ module with $G_{F_v}$ action, $\beta$ a basis of $V_A$, and $\iota_A:V_A/m_A\cong V_\bF$ an isomorphism of $G_{F_v}$-modules such that $\iota_A(\beta)=\beta_\bF$. This data is the same as a homomorphism $\rho_A:G_{F_v}\to \GL_2(A)$ lifting $\overline{\rho}|_{G_{F_v}}$. This functor is pro-representable by a complete local Noetherian $\cO$-algebra $R^{\Box}_v$. The subfunctor $D_v^{\psi}$ consisting of lifts with determinant $\psi$ is pro-represented by a quotient $R^{\psi,\Box}_v$ of $R^\Box_v$. A \emph{deformation condition at $v$} is a relatively representable subfunctor $\overline{D}^{\psi}_v\subset D^{\psi}_v$ satisfying the dimension conditions in \cite[5.4]{Boecklerbook}  \mar{to be streamlined}. The condition of being in $\overline{D}^{\psi}_v(A)$ is assumed to not depend on the choice of basis of $V_A$. When $v|2$, we also consider some other subfunctors of $D_v\widehat{\otimes}_\mathcal{O} \Lambda(G_{F_v})$ on $\mathfrak{Ar}_{\Lambda(G_{F_v})}$ as in  \cite[1.4.3]{Allen2014}. 

A \emph{deformation problem} is the data of $(\overline{\rho},F,\Sigma\subset S,(\overline{D}^{\psi}_v)_{v\in \Sigma} )$. Given such data, let $\overline{D}_{F,\Sigma,S}^{\psi,\Box}$ be the functor which assigns to $(A,m_A)\in \mathfrak{Ar}_\cO$ the set of isomorphism classes of tuples $(V_A,\iota_A,(\beta_v)_{v\in \Sigma})$, where:
\begin{itemize} \item $V_A$ is a free $A$-module with $G_{F,S}$ action and $\iota_A: V_A/m_A \cong V_\bF$ is an isomorphism of $G_{F,S}$-modules.
\item $\beta_v$ is a basis of $V_A$ such that $\iota_A(\beta_v)=\beta_\bF$. 
\item The lifting of $\overline{\rho}|_{G_{F_v}}$ determined by $(V_A,\iota_A,\beta_v)$ (viewed as a $G_{F_v}$-module by restriction) is in $\overline{D}_v^{\psi}(A)\subset D_v(A)$.
\item The determinant of $V_A$ is given by $\psi$.
\end{itemize}
The functor $\overline{D}_{F,\Sigma,S}^{\psi,\Box}$ is pro-representable and we denote by $\overline{R}^{\psi,\Box}_{F,\Sigma,S}$ the corresponding deformation ring. We define the functor $\overline{D}^{\psi}_{F,\Sigma,S}$ in exactly the same way as $\overline{D}_{F,\Sigma,S}^{\psi,\Box}$, except that we do not add the data of $(\beta_v)$. Because $\overline{\rho}$ is absolutely irreducible, the functor $\overline{D}^\psi_{F,\Sigma,S}$ is pro-representable and we denote by $\overline{R}^{\psi}_{F,\Sigma,S}$ the corresponding deformation ring.

The $E$-points of $\mathrm{Spec}\, \overline{R}^{\psi}_{F,\Sigma,S}$ is precisely the set of deformations $\rho$ of $\overline{\rho}$ to $\cO$ with determinant $\psi$ such that for each $v\in\Sigma$, $\rho_{G_{F_v}}$ satisfies the deformation condition $\overline{D}^\psi_v$. The problem of simultaneous level raising will be reduced to showing that this set is non-empty for suitable deformation conditions. 

Let $\delta= \dim_\bF \mathrm{Ker}(H^0(G_{F,S},(\mathrm{ad}^0\rhobar)^{*})\to \oplus_{v\in S\setminus \Sigma} H^0(G_{F_v},(\mathrm{ad}^0\rhobar)^{*}))$, where the superscript $*$ means Pontryagin dual. Note that $\delta=0$ if $S\setminus \Sigma\neq \emptyset$.
We have the following estimate \cite[Theorem 5.4.1]{Boecklerbook}:
\begin{thm} \label{Krull dim} If $\delta=0$, then $\dim \overline{R}^\psi_{F,\Sigma,S}\geq 1$.
\end{thm}

Let us now assume that our deformation problem is of the following form:
\begin{itemize}
\item For $v|\infty$: we let $\overline{D}^\psi_v$ be the subfunctor represented by the quotient of $R^\Box_v$ which is cut out by the equation $\det(\rho(c_v)-X)=X^2-1$, where $c_v$ is the complex conjugation. That is, we look at odd deformations.
\item For $v|2$: Assume $\rhobar$ has a $G_{F_v}$-stable line $L$ such that $G_{F_v}$ acts on $V_\bF/L$ via a character $\overline{\chi}$, and that $\psi$ is a ramified character. Then there is a unique $\cO$-flat quotient $\tilde{R}^{\psi,\Box}_v$ of $R^\Box_v$ such that for any finite extension $E'$ of $E$, an $E'$-point $x$ of $\mathrm{Spec}\, R^\Box_v$ with corresponding representation $\rho_x$ factors through this quotient if and only if $\det \rho_x=\psi$, and $\rho_x$ has a Galois-stable line $L\subset V_x$ such that the Galois action on $V_x/L$ is unramified. In the notation of \cite[1.4.3]{Allen2014}, this is (the $\cO$-flat quotient with same generic fiber of)  $R^{\Delta,\psi}_{\Lambda(G_{F_v})}\otimes _{\Lambda'(G_{F_v})}\cO$, where the homomorphism $\Lambda'(G_{F_v})\to \cO$ is the specialization homomorphism from the character $I_{F_v}\to \Lambda'(G_{F_v})^\times$ to the trivial character. Note that a priori, $R^{\Delta,\psi}_{\Lambda(G_{F_v})}\otimes _{\Lambda'(G_{F_v})}\cO$ need not be a quotient of $R^{\Box}_v$, because it keeps track of the character Galois acts on the line $L$. However, with the the assumption that $\psi$ is ramified, the line $L$ is uniquely determined by the deformation, and hence it is in fact a quotient of $R^{\Box}_v$.
We let $\overline{D}_v^\psi$ be the subfunctor represented by an $\cO$-torsion free quotient of $R^\Box_v$ corresponding to an irreducible component of $\tilde{R}^{\psi,\Box}_v[\frac{1}{2}]$, and  let $\overline{D}_v^{\mathrm{big},\psi}$ be the functor represented by the quotient $R^{\Delta,\psi}_{\Lambda(G_{F_v})}$ of $R^{\psi,\Box}_v\widehat{\otimes}_{\cO} \Lambda(G_{F_v})$ in the category $\mathfrak{Ar}_{\Lambda(G_{F_v})}$. When $\psi=\psi_2$ is the cyclotomic character, the generic fiber of $\tilde{R}^{\psi,\Box}_v$ consists of the following three types of irreducible components: components whose generic $E'$-point gives rise to an extension of the trivial character by $\psi_2$, a quadratic unramified twist of an extension of the trivial character by $\psi_2$, or a crystalline ordinary representation. There is at most one component of the first two types, and possibly more than one component of the third type. This fact, and the fact that $\tilde{R}^{\psi,\Box}_v$ satisfy the dimension requirement of \cite[5.4]{Boecklerbook} follow from the arguments in \cite[2.4]{Kisin2009} and \cite[4.1-4.3]{Snowden}.

\item For $v\in \Sigma$, $v\nmid 2$: Let $R^{\psi,\Box}_v$ be the ring pro-representing the subfunctor of $D_v$ classifying lifts of fixed determinant $\psi$. It is known (\cite[3]{Boecklerbook}, \cite{Pilloni}) that $R^{\psi,\Box}_v[\frac{1}{2}]$ is equidimensional of dimension 3, with smooth irreducible components. The deformation conditions we take are those given by a choice of (union of) irreducible components, that is the subfunctor represented by the unique $\cO$-torsion free quotient of $R^{\psi,\Box}_v$ whose generic fiber is the chosen (union of) components. On each irreducible component, the inertial Weil-Deligne type is constant. Either there is no irreducible component whose inertial Weil--Deligne type is $(1\oplus 1, N\neq 0)$, or there are exactly two of them, which differ by an unramified quadratic twist. In the case these components correspond to the Steinberg representation and its unramified quadratic twist, we call them the \emph{Steinberg component} and the \emph{twisted Steinberg component}.

\end{itemize}
If $F'$ is a totally real finite extension such that $\overline{\rho}|_{G_{F'}}$ is still absolutely irreducible, one can consider the ``base change'' deformation problem by replacing $S$, $\Sigma$ with the set $S'$, $\Sigma'$ of primes in $F'$ above them, and restricting the inertial Weil--Deligne types. If we denote by $\overline{R}^{\psi}_{F',\Sigma',S'}$ the corresponding deformation ring, the argument in Lemma 1.2.3 of \cite{BLGGT} shows that $\overline{R}^{\psi}_{F,\Sigma,S}$ a finite $\overline{R}^{\psi}_{F',\Sigma',S'}$-algebra.
\begin{theorem} \label{Prescribed type} Let $(\overline{\rho},F,\Sigma\subset S,(\overline{D}^{\psi}_v)_{v\in \Sigma} )$ be a deformation problem as above. Assume:
\begin{itemize}
\item $\psi \psi_2^{-1}$ is a finite order character, where $\psi_2$ is the 2-adic cyclotomic character.
\item $S\setminus \Sigma\neq \emptyset$.
\item  For $v\in S\setminus \Sigma$, assume that no component of $R_v^{\psi,\Box}$ has inertial Weil--Deligne type with $N\neq 0$.
\item  For $v\in \Sigma$, assume that $\overline{D}^{\psi}_v$ is given by one component of $R_v^{\psi,\Box}[\frac{1}{2}]$ (or of $\tilde{R}^{\psi,\Box}_v[\frac{1}{2}]$ when $v|2$).
\item $\mathrm{Im}\overline{\rho}$ is dihedral, induced from a quadratic extension $K/F$.
\item If $K$ is imaginary CM, then there is a prime $v|2$ of $F$ which does not split in $K$. 
\end{itemize} 
Then  $\overline{R}^\psi_{F,\Sigma,S}[\frac{1}{2}]\neq 0$. In particular, there is a deformation $\rho:G_{F,S}\to \GL_2(E')$ satisfying the deformation conditions of the deformation problem, with $E'$ a finite extension of $E$. Furthermore, all such deformations are modular.
\end{theorem}
\begin{proof} We indicate how to deduce this from the main result of \cite{Allen2014}. First we replace $F$ by a solvable totally real extension $F'$ as in the proof of Theorem 5.2.1 of loc. cit. This has the effect of making the assumptions in Section 4.4 there hold.
Let us denote by $\overline{R}^{\mathrm{big},\psi}_{F',S'}$ the deformation ring defined as in \cite[p.1316]{Allen2014}. It is the deformation ring representing the deformation functor $\overline{D}^{\mathrm{big},\psi}_{F',S',S'}$, which is defined as the deformation problem in the category of local Artinian $\Lambda(G_{F'})=\hat{\otimes}_{v|2}\Lambda(G_{F'_v})$-algebra such that the local deformation condition at $v|2$ is given by $\overline{D}_v^{\mathrm{big},\psi}$, and the local deformation conditions at the other places are given by the component that contains the image under the map of local lifting rings of the components in the definition of $\overline{R}^\psi_{F,\Sigma,S}$. By enlarging $F'$ if needed, the local deformation conditions at the finite places in $\Sigma'$ obtained this way is either the unramified or an unramified twist of Steinberg component, while the deformation conditions at the places in $S'\setminus \Sigma'$ become the unramified component. This shows that we are indeed in the setting of loc. cit. Let $\Lambda'(G_{F'})=\hat{\otimes}_{v|2}\Lambda'(G_{F'_v})$. There is a homomorphism $\phi:\Lambda'(G_{F'})\to \cO$ (``weight 2 specialization'') such that there is a surjection $\overline{R}^{\mathrm{big},\psi}_{F',S'}\otimes_{\Lambda'(G_{F'}),\phi}\cO\onto \overline{R}^{\psi}_{F,\Sigma',S'} $.

Now by Proposition 4.4.3 of \cite{Allen2014}, every prime of $\overline{R}^{\mathrm{big},\psi}_{F',S'}$ is pro-modular, and hence $(\overline{R}^{\mathrm{big},\psi}_{F',S'})^{\mathrm{red}}$ is identified with a localized Hida Hecke algebra of $F'$. In particular $(\overline{R}^{\mathrm{big},\psi}_{F',S'})^{\mathrm{red}}$ is a finite $\Lambda'(G_{F'})$-algebra. Because $\overline{R}^{\mathrm{big},\psi}_{F',S'}$ is Noetherian, $\overline{R}^{\mathrm{big},\psi}_{F',S'}$ is also a finite $\Lambda'(G_{F'})$-algebra. But this implies $\overline{R}^{\psi}_{F',\Sigma',S'}$ is a finite $\cO$-algebra, and hence $\overline{R}^{\psi}_{F,\Sigma,S}$ is a finite $\cO$-algebra. Because $\dim \overline{R}^{\psi}_{F,\Sigma,S} \geq 1$ by Theorem \ref{Krull dim}, this forces $\overline{R}^{\psi}_{F,\Sigma,S}[\frac{1}{2}]\neq 0$. The residue fields at each maximal ideals of this ring are finite extensions of $E$, whose points give rise to the desired characteristic 0 deformations $\rho$. Furthermore, the argument above shows that after restriction to $F'$, any such $\rho$ comes from the specialization at weight $2$ of a Hida Hecke algebra, and hence $\rho|_{G_{F'}}$ is modular. By solvable base change, $\rho$ is also modular.
\end{proof}
We now apply this to prove Theorem \ref{thm:levelraising} when $E$ is ordinary at 2. We choose our deformation problem as follows:
\begin{itemize}
\item $\overline{\rho}=\overline{\rho}_{E,2}:G_\QQ\to \GL_2(\FF_2)$ is the mod 2 representation of the elliptic curve $E$.
\item $\psi=\psi_2$ is the 2-adic cyclotomic character.
\item $\Sigma$ consists of the finite primes dividing $Nq_1\cdots q_m$ as well as $\infty$.
\item For $p||N$ and $p\neq 2$, the deformation condition $\overline{D}^{\psi}_p$ is given by the Steinberg component (resp. twisted Steinberg component) if $\epsilon_p=+1$ (resp. $\epsilon_p=-1$).
\item For $p^2|N$, the deformation condition $\overline{D}^{\psi}_p$ is given by the unique component of $R^{\psi,\Box}_v[\frac{1}{2}]$ that contains $\rho_{E,2}|_{G_{\QQ_p}}$.
\item At $v|2\infty$, the deformation conditions are chosen to be $\overline{D}^\psi_v$ as below Theorem \ref{Krull dim}. Note that we are in the situation dealt there because $E$ was assumed to be ordinary at 2. For $v|2$, we choose the crystalline ordinary component if $E$ has good reduction at $2$. If $E$ has multiplicative reduction at $2$, we choose the component that either contains $\rho_{E,2}|_{G_{\QQ_2}}$ or contains its unramified quadratic twist depending on the chosen sign $\epsilon_2$.
\item At $q_i$, the deformation condition is given by either the Steinberg component or the twisted Steinberg component, depending on the sign $\varepsilon_i$.
\item $S=\Sigma\cup \{q_0\}$ where $q_0$ is an auxiliary prime as in Definition \ref{Auxiliary prime definition}.
\end{itemize}

By Remark \ref{rem:2splits}, we know that $2$ does not split in $K=\mathbb{Q}(\sqrt{\Delta})$. This together with Lemma \ref{Auxiliary prime} shows that the hypotheses of Theorem \ref{Prescribed type} holds. Thus we get a modular deformation of $\overline{\rho}$ which corresponds to a weight 2 newform $g=\sum b_n q^n$ (with associated automorphic representation $\pi=\otimes \pi_p$) such that:
\begin{itemize}
\item $\pi$ has trivial central character.
\item For $p|2N$, the conductor of $\pi_p$ is equal to $\ord_p(N)$. If $p||N$, $\pi_p$ is Steinberg or the unramified quadratic twist of Steinberg depending on $a_p=1$ or $-1$, thus $b_p=a_p$ for such $p$.
\item For $i>0$, $\pi_{q_i}$ is Steinberg or the unramified quadratic twist of Steinberg depending on $\varepsilon_i=1$ or $-1$. Thus $b_{q_i}=\varepsilon_i$.
\end{itemize}

\subsection{Proof of Theorem \ref{thm:levelraising} in the ordinary case}\label{sec:ordinarycaseproof}
By construction, $q_0\equiv3\bmod4$ and $\overline{\rho}(\Frob_{q_0})$ has order 3 (hence has distinct eigenvalues). Applying Corollary \ref{Level raising form at auxiliary prime}, it follows that our level raised form $g$ is either unramified at $q_0$ or its twist $g\otimes \chi_{q_0}$ is unramified at $q_0$, where $\chi_{q_0}$ is the unique quadratic character that is ramified at $q_0$ and unramified everywhere else. In the former case, the form $g$ satisfies the conclusion of Theorem \ref{thm:levelraising}. In the latter case, $g\otimes \chi_{q_0}$ has the desired conductor, so we only need to check the matching of the signs at primes $p$ where the conductor is 1. But such a prime $p$ either satisfy $p||N$ or $p=q_i$ ($i>0$). Since twisting by $\chi_{q_0}$ changes the sign $\varepsilon$ at $p$ to $\varepsilon \chi_{q_0}(\Frob_p)=\varepsilon \legendre{p}{q_0}=\varepsilon$ if $p\neq p_1$, we see that $g\otimes \chi_{q_0}$ satisfies the conclusion of Theorem \ref{thm:levelraising}. This finishes the proof in the ordinary case.

\section{Simultaneous level raising: supersingular case}
\label{Simultaneous level raising: supersingular case}
Let $D$ be a quaternion algebra over $\QQ$. We denote by $G_D$ the $\QQ$-algebraic group $D^\times$, $Z\cong \GG_m$ its center and $\Sigma(D)$ the set of primes where $D$ is ramified. Assume $2\notin \Sigma(D)$. Let $\nu_D:G_D\to \GG_m$ be the reduced norm map. Fix a maximal order $\cO_D$ of $D$, and fix once and for all an isomorphism between $\cO_D\otimes \mathbb{Z}_p\cong M_2(\ZZ_p)$ for each place $p\notin \Sigma(D)$. This determines an isomorphism $G_D(\QQ_p)\cong \GL_2(\QQ_p)$.

Given an open subgroup $U$ of $G_D(\mathbb{A})$ of the form $\prod U_p$, such that the set $S$ of primes such that $U_p\neq \GL_2(\ZZ_p)$ is finite, we have the abstract Hecke algebra $\TT=\ZZ[\{T_p, S_p\}_{p\notin S}]$ (this depends on $U$ through the set $S$, though this dependence is not in the notation).  A maximal ideal $\mathfrak{m}\subset \TT$ is called \emph{Eisenstein} if there exists some positive integer $d$ such that $T_p-2\in \mathfrak{m}$ for all but finitely many primes $p=1 \bmod d$.

Let $\Iw_1(p^n)$ (resp., $\Iw(p^n)$) be the subgroup of $\GL_2(\mathbb{Z}_p)$ consisting of matrices which are upper triangular unipotent (resp. upper triangular) mod $p^n$. If $U=\prod U_p$ is an open subgroup, and $p$ is a prime such that $U_p=\GL_2(\ZZ_p)$. We denote by $U_0(p)$ the open subgroup of $U$ which agrees with $U$ away from $p$, and $U_0(p)_p=\Iw(p)\subset U_p$.

\subsection{Quaternionic forms: definite case}  \label{sec:defcase}
Throughout this section assume $D$ is definite. As in \cite{Allen2014}, for each $\Sigma'\subset \Sigma(D)$, a $(\Sigma'\subset \Sigma(D))$-open subgroup $U\subset G_D(\AA^{\infty})$ is a subgroup of the form $U=\prod_p U_p$ such that:
\begin{itemize}
\item $U_p\subseteq \GL_2(\ZZ_p)$ for $v\notin \Sigma(D)$, via our chosen identification. Equality holds for almost all $p$.
\item $U_p=G_D(\QQ_p)=D_p^\times$ for $p\in \Sigma'$.
\item $U_p=(\cO_D \otimes \mathbb{Z}_p)^\times$ for $p\in \Sigma(D)\setminus \Sigma'$.
\end{itemize}
Note that an $(\emptyset\subset \Sigma(D))$-open subgroup is an open compact subgroup. 

Let $\gamma=(\gamma_p)_{p\in \Sigma'}$ be a tuple of unramified characters $\gamma_p:\QQ_p^\times \to \mu_2$. This determines a character of $\gamma:G_D(\AA^{\infty})\to \mu_2$ given by composing the projection to $\prod_{p\in \Sigma'} G_D(\QQ_p)$ and $\prod_{p\in \Sigma'} \gamma_p\circ \nu_D$. If $A$ is a topological $\ZZ_2$-module, define $S_{\gamma}(U,A)$ to be the space of functions
\[f: G_D(\QQ)\setminus G_D(\AA^{\infty})/UZ(\AA^{\infty})\to A\]
such that $f(gu)=\gamma(u)f(g).$
Because $D$ is definite, there exists $t_1,\cdots t_n\in G_D(\AA^{\infty})$ such that $G_D(\QQ)\setminus G_D(\AA^{\infty})/UZ(\AA^{\infty})=\coprod G_D(\QQ)t_iUZ(\AA^{\infty})$, and this gives the identification
\[S_{\gamma}(U,A)\cong \oplus_{i=1}^n A^{\gamma((UZ(\AA^{\infty})\cap t_i^{-1}G_D(\QQ)t_i)/Z(\mathbb{Q}))}.\]
Here we view $A$ as a $\mu_2$-module via its $\ZZ_2$-module structure, and the superscript means taking invariants.
In particular, if $(UZ(\AA^{\infty})\cap t_i^{-1}G_D(\QQ)t_i)/Z(\mathbb{Q})=1$ (or if $\gamma$ is trivial) then $S_{\gamma}(U,A)=S_{\gamma}(U,\ZZ_2)\otimes_{\ZZ_2} A$. Without any assumption on $U$, this holds if $A$ is $\ZZ_2$-flat.

\begin{lem}  \label{Sufficiently small}Fix a prime $p\notin \Sigma(D)$. Let $U$ be a $(\Sigma'\subset \Sigma(D))$-open subgroup.  If $U_p\subset \Iw_1(p^n)$ for $n$ large enough (depending only $p$), then $UZ(\AA^{\infty})\cap t^{-1}G_D(\QQ)t)/Z(\mathbb{Q})=1$ for any $t\in G_D(\AA^\infty)$.
\end{lem}
\begin{proof} This is Lemma 2.1.5 of \cite{Allen2014}.
\end{proof}

\begin{definition}\label{def:sufsmall}
We call a subgroup $U$ satisfying the conclusion of the lemma \textit{sufficiently small}.  
\end{definition}

Let $S$ be the set of primes such that $U_p\neq \GL_2(\ZZ_p)$. The abstract Hecke algebra $\TT=\ZZ[\{T_p, S_p\}_{p\notin S}]$ acts on $S_{\gamma}(U, A)$ through the usual double coset operators $T_p$, $S_p$. Denote by $\TT(\gamma,U,A)$ the quotient of $\TT$ that acts faithfully on $S_{\gamma}(U,A)$. The subspace $S_{\gamma}(U,A)^{\mathrm{triv}}$ consisting of functions that factor through $\nu_D$ is stable under $\TT$. Fix an embedding $\mathbb{Q}_2\hookrightarrow \mathbb{C}$. The Jacquet--Langlands correspondence gives a $\TT$-equivariant isomorphism 
\[(S_\gamma(U,\ZZ_2)/S_\gamma(U,\ZZ_2)^{\mathrm{triv}}) \otimes_{\ZZ_2} \CC \cong \bigoplus \pi^V. \]
Here $V\subseteq \GL_2(\AA^\infty)$ is the open compact subgroup such that  $V_p=U_p$ if $p\notin \Sigma(D)$, and $V_p=\Iw(p)$ if $p\in \Sigma(D)$. The sum runs over $\pi$ such that 
\begin{itemize} 
\item $\pi$ is an algebraic automorphic representations of $\GL_2(\AA)$ such that $\pi_\infty$ is discrete series with trivial infinitesimal character (i.e. $\pi$ corresponds to a modular form of weight 2) and trivial central character.
\item For $p\in \Sigma(D)$, the local representation $\pi_p$ is an unramified twist of the Steinberg representation $\mathrm{St}$. If $p\in \Sigma'$, then $\pi_p\cong \gamma_p\otimes \mathrm{St}$.
\end{itemize}
It follows that $\TT(\gamma,U,\ZZ_2)\otimes \overline{\QQ}_2$ is a product of fields, and each homomorphism $$\TT\onto \TT(\gamma,U,\ZZ_2) \to \overline{\ZZ}_2$$ corresponds to the system of Hecke eigenvalues of a modular form of weight $2$ whose automorphic representation satisfies the above condition. We say that such a system of Hecke eigenvalues \emph{occurs in $S_\gamma(U,\overline{\ZZ}_2)$}. Given a maximal ideal $\mathfrak{m}$ of $\TT$ corresponding to a homomorphism $\overline{\theta}:\TT\to \overline{\FF}_2$, there exists a modular form $g$ whose system of Hecke eigenvalues is congruent to $\overline{\theta}$ is equivalent to $\mathfrak{m}$ being in the support of $S_\gamma(U,\overline{\ZZ}_2)/S_\gamma(U,\overline{\ZZ}_2)^{\mathrm{triv}}$, or equivalently $(S_\gamma(U,\overline{\ZZ}_2)/S_\gamma(U,\overline{\ZZ}_2)^{\mathrm{triv}})_\mathfrak{m}\neq 0$. If $\mathfrak{m}$ correspond to the mod 2 reduction of a system of Hecke eigenvalues of a modular form, then $\mathfrak{m}$ is Eisenstein if and only if the associated mod 2 representation $\rho_\mathfrak{m}:G_\mathbb{Q}\rightarrow \GL_2(\mathbb{F}_2)$ is reducible (\cite[Prop. 2]{Diamond1994}). One knows that the Hecke action on $S_\gamma(U,\overline{\ZZ}_2)^{\mathrm{triv}}$ is Eisenstein, and hence if $\mathfrak{m}$ is non-Eisenstein, $(S_\gamma(U,\overline{\ZZ}_2)/S_\gamma(U,\overline{\ZZ}_2)^{\mathrm{triv}})_\mathfrak{m}\neq 0$ is equivalent to $S_\gamma(U,\overline{\ZZ}_2)_\mathfrak{m}\neq 0$.

Now suppose $U$ is a $(\Sigma'\subset \Sigma(D))$-open subgroup and $p$ is a prime such that $U_p=\GL_2(\ZZ_p)$. Recall that $U_0(p)$ the $(\Sigma'\subset \Sigma(D))$-open subgroup of $U$ which agrees with $U$ away from $p$, and $U_0(p)_p=\Iw(p)\subset U_p$. We say that a system of Hecke eigenvalues in $\overline{\ZZ}_2$ that occurs in $S_\gamma(U_0(p),\overline{\ZZ}_2)$ but not in $S_\gamma(U,\overline{\ZZ}_2)$ is \emph{$p$-new}. Under the Jacquet--Langlands correspondence, it corresponds to an automorphic representation whose component at $p$ is an unramified twist of the Steinberg representation.

We have the following level raising result:
\begin{lemma} \label{Level raising definite} Let $U$ be an $(\emptyset\subset \Sigma(D))$-open subgroup that is sufficiently small, and $U_p=\GL_2(\ZZ_p)$. Suppose $\mathfrak{m}$ is a maximal ideal of $\TT$ in the support of $S(U,\ZZ_2)$, and that $T_p \in \mathfrak{m}$. Then there exists a $p$-new system of Hecke eigenvalues lifting $\mathfrak{m}$.
\end{lemma}
\begin{proof} This is a reformulation of Lemma 3.3.3 of \cite{Kisin2009}.
\end{proof}
\subsection{Quaternionic forms: indefinite case}
Throughout this section $D$ is assumed to be indefinite and not split. Let $U\subset G_D(\AA^\infty)$ be an open compact subgroup. The double coset space
\[G_D(\QQ)\setminus \HH^{\pm}\times G_D(\AA^\infty)/U\]
is naturally the complex points of an algebraic curve $X_U$, which is in fact defined over $\QQ$. 

Following \cite{Diamond1994}, for $N$ not divisible by any prime of $\Sigma(D)$, let $V_1(N)$ denote the open compact subgroup such that 
\begin{itemize}
\item  For $p\in \Sigma(D)$, $V_1(N)_p=(\cO_D\otimes \ZZ_p)^\times$.
\item For $p|N$, $V_1(N)_p\subseteq \GL_2(\ZZ_p)$ consists of matrices whose mod $p$ reduction is $\left(\begin{smallmatrix} * & * \\0 & 1\end{smallmatrix}\right) $.
\item $V_1(N)_p=\GL_2(\ZZ_p)$ otherwise.
\end{itemize}
Then we say $U$ is \emph{sufficiently small} if $U\subset V_1(N)$ for some $N\geq 4$.  If $U$ is sufficiently small, then $X_U$ is naturally the moduli space of false elliptic curves $(A,i)$ with level structure (see \cite[Section 3, 4]{Diamond1994}).

For the rest of this section, we will let $U=V_1(q)$ for some suitable prime $q>3$. In particular, such $U$ is sufficiently small, and that $\nu_D(U)=\hat{\ZZ}^\times$. Suppose $p$ is a prime away from $\Sigma(D)\cup \{q\}$. There are two natural \'{e}tale projection maps $\pi_1, \pi_2: X_{U_0(p)}\to X_U$, which gives the Hecke correspondence $T_p$ at $p$. The abstract Hecke algebra $\TT$ consisting of Hecke operators $T_l$, $S_l$ (for $l$ such that $U_0(p)_l=\GL_2(\ZZ_l)$) acts on the whole situation by \'{e}tale correspondences, and hence induces endomorphisms on \'{e}tale cohomology groups. Because $U_0(p)$ has full level at $2$, this picture makes sense over $\ZZ_2$. 
\mar{$S_p$=diamond operators. They occur because of the $V_1(N)$ level structure}
We have the following diagram:
\[\xymatrix{ H^1_{\et}(X_U,\ZZ_2)^2 \ar[r]^{i^*} & H^1_{\et}(X_{U_0(p)},\ZZ_2) \ar[r]^{i_*} &H^1_{\et}(X_U,\ZZ_2)^2} \]
where $i^*=\pi_1^*+\pi_2^*$ and $i_*=\pi_{1*}+\pi_{2*}$. One computes that the composition $i_*i^*$ has the form 
\[\begin{pmatrix} p+1 & T_p \\ S_p^{-1}T_p &p+1 \end{pmatrix}.\]
We have the following facts: 
\begin{itemize}
\item $H^1_{\et}(X_U,\ZZ_2)$, $H^1_{\et}(X_{U_0(p)},\ZZ_2)$ are torsion-free, and carry a perfect alternating pairing given by Poincare duality.
\item $i_*$ is the adjoint of $i^*$ with respect to the pairings.
\item $i^*$ is injective after inverting $2$.
\end{itemize}
By the Jacquet--Langlands correspondence, the system of Hecke eigenvalues $\TT\to \overline{\ZZ}_2$ that occurs in $H^1_{\et}(X_U,\ZZ_2)$ are exactly those of automorphic representations $\pi$ of $\GL_2(\AA)$  such that
\begin{itemize}
\item $\pi_\infty$ is discrete series with trivial infinitesimal character (i.e. $\pi$ corresponds to a modular form of weight 2).
\item For $l\in\Sigma(D)$, $\pi_l$ is an unramified twist of the Steinberg representation.
\item $\pi^{U}\neq 0$.
\end{itemize} 
A system of Hecke eigenvalues that occurs in the cokernel of $i^*$ will correspond exactly to an automorphic representation $\pi$ as above which is furthermore \emph{$p$-new}, i.e. that $\pi_p$ is an unramified twist of the Steinberg representation. For a maximal ideal $\mathfrak{m}$, $\pi$ contributes to $H^1_{\et}(X_{U},\ZZ_2)_\mathfrak{m}$ if and only if the system of Hecke eigenvalues of $\pi$ is congruent to the one given by $\mathfrak{m}$.

The following fact (``Ihara's lemma'') is the key input to level raise in this setting:
\begin{lemma} \label{lem:Ihara} Let $\mathfrak{m}$ be a non-Eisenstein maximal ideal of $\TT$ which comes from a modular form $g$ of level prime to 2. Assume that $\rho_\mathfrak{m}|_{G_{\QQ_2}}$ is supersingular. Then the localization of $i^*$ at $\mathfrak{m}$ is injective mod 2.
\end{lemma}
\begin{proof} In \cite{Diamond1994}, \cite{Diamond1994a}, this is proven for $\ell\geq 3$. This restriction comes from the fact that they use Fontaine--Laffaille theory. We show how to adapt their argument to our situation. If $X$ is a smooth proper curve, we let $J(X)$ denote its Jacobian. The 2-divisible groups $J(X_U)[2^\infty]$ and $J(X_{U_0(p)})[2^\infty]$ admit direct summands $J(X_U)[2^\infty]_\mathfrak{m}$ and $J(X_{U_0(p)})[2^\infty]_\mathfrak{m}$, which are stable under $\TT$. By the Eichler--Shimura relations, the Galois representation on $T_2J(X_U)[2]_\mathfrak{m}$ and $T_2J(X_{U_0(p)})[2]_\mathfrak{m}$ are successive extensions of $\rho_\mathfrak{m}$, and hence the summands are connected and unipotent $2$-divisible groups. Consider the map $J(X_{U_0(p)})[2^\infty]_\mathfrak{m} \to J(X_U)[2^\infty]_\mathfrak{m}^2$ which on the Tate module is dual to $i^*$. Assume that $i^*$ is not injective mod 2. Then the induced map $J(X_{U_0(p)})[2]_\mathfrak{m} \to J(X_U)[2]_\mathfrak{m}^2$ is not surjective, and hence has a cokernel that is a successive extension of $\rho_\mathfrak{m}$. 

By Fontaine's theorem (see \cite[Theorem 7.2.10]{Brinon-Conrad}), it follows that the induced map on the Honda systems attached to $J(X_U)[2^\infty]_\mathfrak{m}^2$ and $J(X_{U_0(p)})[2^\infty]_\mathfrak{m}$ is not injective mod 2. But this implies (since the $2$-divisible groups involved are connected) that the induced map $$H^0(J(X_{U}),\Omega^1)_\mathfrak{m}^2\cong \mathrm{Lie}(J(X_U)[2^\infty]_\mathfrak{m}^2)^* \to \mathrm{Lie}(J(X_{U_0(p)})[2^\infty]_\mathfrak{m})^*\cong H^0(J(X_{U_0(p)}),\Omega^1)_\mathfrak{m}$$ is not injective mod 2. Thus
\[\pi_1^*+\pi_2^*:H^0(X_U\otimes \FF_2, \Omega^1)_\mathfrak{m}^2 \to H^0(X_{U_0(p)}\otimes \FF_2, \Omega^1)_\mathfrak{m}\]
has non-trivial kernel. Let $(\omega_1,\omega_2)$ be a non-zero element in the kernel. Arguing as in \cite[Lemma 8 and 9]{Diamond1994}, we conclude that the divisor of $\omega_1$ must be inside the supersingular locus of $X_U\otimes \FF_2$, and in fact contains all supersingular points. Now by \cite[Section 5]{Kassaei1999}, there is a line bundle $\omega$ on $X_U\otimes \FF_2$ and a section $\mathrm{Ha}\in H^0(X_U\otimes \FF_2, \omega)$ which vanishes to order 1 at each supersingular points (Note that even though the running assumption of \cite{Kassaei1999} is that $\ell>3$, this is not needed for \cite[Section 5]{Kassaei1999}). This property determines $\mathrm{Ha}$ up to a non-zero scalar.  Using this characterization, we get $\pi_1^*\mathrm{Ha}$ and $\pi_2^*\mathrm{Ha}$ coincide up to a non-zero scalar. It is known that $\TT$ acts on $\mathrm{Ha}$ through an Eisenstein maximal ideal.

Now, we have an isomorphism $\Omega^1\cong \omega^{\otimes 2}$, and hence we conclude that $\omega_1=\mathrm{Ha}\cdot\omega_1'$ for some $\omega_1'\in H^0(X_U,\omega)$, and similarly for $\omega_2$. But now in the equation $\pi_1^*\omega_1=-\pi_2^*\omega_2$ we can cancel out $\pi_1^*\mathrm{Ha}$, and hence $\pi_1^*\omega_1'$ agrees with $\pi_2^*\omega_2'$ up to a non-zero scalar. Repeating the argument now forces $\omega_1'=c\cdot\mathrm{Ha}$, and hence $\omega_1=c\cdot\mathrm{Ha}^2$. But then the action of $\TT$ on $\omega_1$ is Eisenstein, contradicting the fact that $\omega_1\in H^0(X_{U}\otimes \FF_2, \Omega^1)_\mathfrak{m}$ and $\mathfrak{m}$ is non-Eisenstein.
\end{proof}

The following is the main result of this section:
\begin{lem} \label{Level raising indefinite} Suppose $\mathfrak{m}$ is a maximal ideal of $\TT$ that corresponds to the mod 2 reduction of a system of Hecke eigenvalues that contributes to $H^1_{\et}(X_U,\ZZ_2)$. Assume that $\mathfrak{m}$ is non-Eisenstein, that $\rho_\mathfrak{m}$ is supersingular at 2 and that $p$ is a prime such that $T_p\in \mathfrak{m}$. Then there exists a $p$-new system of Hecke eigenvalues lifting $\mathfrak{m}$.
\end{lem}
\begin{proof} By what we have said so far, we only need to show that $H^1_{\et}(X_{U_0(p)},\ZZ_2)_\mathfrak{m}/\mathrm{Im}\,i^*$ is not torsion. Suppose this is the case, then because $i^*$ is injective mod 2 (Lemma \ref{lem:Ihara}), this quotient is actually trivial, and hence $i^*$ is an isomorphism. By duality $i_*$ is also an isomorphism. But $T_p\in \mathfrak{m}$ implies that $i_*i^*$ is 0 mod $\mathfrak{m}$, hence can not be surjective. This gives the desired contradiction.  
\end{proof}
\mar{Need to fix notation for etale cohomology: Need to base change to $\overline{\QQ}_2.$ Maybe define more clearly what system of Hecke eigenvalues mean. Define missing notations}
\subsection{Proof of Theorem \ref{thm:levelraising} in the supersingular case}
In this section we prove Theorem \ref{thm:levelraising} under the assumption that the modular mod 2 representation $\overline{\rho}:G_\QQ\to \GL_2(\FF_2)$ is supersingular at 2. We choose a prime $p_1|N$ such that $\ord_{p_1}(N)>1$ and $\ord_{p_1}(\Delta)$ is odd if such a prime exists, otherwise choose any prime $p_1|N$. Choose an auxiliary prime $q>7$ as in Definition \ref{Auxiliary prime definition}. We choose $U_q \subset \Iw_1(q^n)\subset \GL_2(\ZZ_q)$ a sufficiently small open compact subgroup as in Definition \ref{def:sufsmall}.

We will first show that we can find a weight 2 modular form $g$ with trivial central character, such that $g$ is new at each prime $q_i$ in our list (without specifying the signs):
\begin{prop} \label{unrefined raising} Assume that we are in the situation of Theorem \ref{thm:levelraising}, with $\overline{\rho}$ supersingular. Then there is a modular form $g$ of weight 2 with corresponding automorphic representation $\pi$ of $\GL_2(\AA)$ such that:
\begin{itemize}
\item $\pi$ has trivial central character.
\item For $p||N$ or $p=q_i$, $\pi_p$ is an unramified twist of the Steinberg representation.
\item For $p|N$, $\pi_p^{\Iw(p^{\ord_p(N)})}\neq 0$.
\item $\pi_q^{U_q}\neq 0$.
\item For all other primes $p$, $\pi_p$ is unramified.
\end{itemize}
\end{prop}
\begin{proof}
This is done by induction on the number $m$ of level raising primes. In the case $m=1$ this follows from Ribet's theorem. Assume that we have found a level raising form $g$ at $m$ primes $q_1,\cdots q_m$, and we wish to add in a prime $q_{m+1}$. The automorphic representation $\pi_g$ is an unramified twist of the Steinberg representations at $q_i$ and the primes $p||N$. Let $D$ be the quaternion algebra that ramifies at exactly these primes (it is definite or indefinite depending on the parity of the size of this set). Let $\mathfrak{m}$ be the maximal ideal of the abstract Hecke algebra that corresponds to $\overline{\rho}$. By assumption it is non-Eisenstein, and its associated mod 2 Galois representation is supersingular at 2. Let $U\subset G_D(\AA)$  be the open compact subgroup given by
\begin{itemize}
\item $U_p=(\cO_D\otimes \ZZ_p)^\times$ for $p\in\Sigma(D)$.
\item $U_p=\Iw(p^{\ord_p(N)})$ for $p|N$, $p\notin \Sigma(D)$.
\item$U_q$ is chosen as above.
\item $U_p=\GL_2(\ZZ_p)$ otherwise.
\end{itemize}

If $D$ is definite, $\pi_g$ contributes to the space $S(U,\overline{\ZZ}_2)$. By Lemma \ref{Level raising definite}, we can find an automorphic representation $\pi'$ corresponding to a weight 2 modular form $g'$ congruent to $g$ mod 2, such that $\pi'^{U_0(q_{m+1})}\neq 0$, that $\pi'$ is new at $q_{m+1}$, and that $\pi'$ has trivial central character. This does the inductive step in this case.

If $D$ is indefinite, $\pi_g$ contributes to the space $H^1_{\et}(X_U\otimes \overline{\QQ}_2, \overline{\ZZ}_2)$. By Lemma \ref{Level raising indefinite}, we can find an automorphic representation $\pi'$ corresponding to a weight 2 modular form $g'$ congruent to $g$ mod 2,  such that $\pi'^{U_0(q_{m+1})}\neq 0$, that $\pi'$ is new at $q_{m+1}$. We claim that $\pi'$ must have trivial central character, or equivalently its associated Galois representation $\rho_{\pi'}$ has determinant $\psi_2$. The fact that $\pi'$ has weight 2 and $\pi'^{U_0(q_{m+1})}\neq 0$ implies that $\det \rho_{\pi'}\psi_2^{-1}$ is a finite order character that is unramified at all $p\neq q$. By our choice of $q$ and Remark \ref{det trick}, $\det \rho_{\pi'}\psi_2^{-1}|_{I_{\QQ_q}}$ has order at most 2. But since $\rho_{\pi'}$ is odd, we must have $\det\rho_{\pi'}\psi_2^{-1}(-1)=1$, where we think of $-1\in \ZZ_q^\times \cong I_{\QQ_q^{\mathrm{ab}}}$ via local class field theory.  But since $q\equiv3 \bmod 4$, $-1$ is a generator of $\ZZ_q^\times/(\ZZ_q^\times)^2$, and thus we conclude that $\det\rho_{\pi'}\psi_2^{-1}$ is unramified at $q$ as well. But since $\QQ$ has class number 1, this forces this character to be trivial, so $\pi'$ indeed has trivial central character. This finishes the inductive step in this case.
\end{proof}
Finally, we show how to modify the signs at level raising primes. Let $\pi$ be the automorphic representation given by Proposition \ref{unrefined raising}.

Now let $\Sigma(D)$ be the set of all primes where $\pi_p$ is an unramified twist of the Steinberg representation. Let $D$ be the quaternion algebra whose finite ramification places are exactly the places in $\Sigma(D)$. 
\mar{don't know how to typeset to make the cases look nice}

\noindent\textbf{Case 1: $D$ is definite}. Let $U$ be the $(\Sigma(D)\subset \Sigma(D))$-open subgroup of $G_D(\AA^\infty)$ such that 
\begin{itemize}
\item For $p|N$, $p\notin \Sigma(D)$, $U_p$ is $\Iw(p^{\ord_p(N)})$.
\item $U_q$ is chosen as above.
\item For $p\notin \Sigma(D)$ and $p\!\!\not| Nq$, $U_p=\GL_2(\ZZ_p)$.
\end{itemize}  
Let $\gamma=(\gamma_v)_{v\in \Sigma(D)}$ be the collection of characters of $\QQ_p^\times\to \mu_2$ such that $\gamma_p(p)=\epsilon_p$, an arbitrarily chosen sign at the prime $p\in \Sigma(D)$. The automorphic representation $\pi$ determines a $\gamma_\pi$, which is the tuple of signs of $\pi_p$ for $p\in \Sigma(D)$. Now since $U$ is sufficiently small, the reduction mod 2 maps
\begin{align*}
S_\gamma(U,\ZZ_2)\to S_\gamma(U,\FF_2), \\
S_{\gamma_\pi}(U,\ZZ_2)\to S_{\gamma_\pi}(U,\FF_2)
\end{align*}
are both surjective. Note however that $S_{\gamma_\pi}(U,\FF_2)= S_{\gamma}(U,\FF_2)$, because any $\gamma$ reduces to the trivial character mod 2. If $\mathfrak{m}$ is the ideal in $\TT$ associated to the mod 2 reduction of the system of Hecke eigenvalues of $\pi$, we see that $ S_{\gamma_\pi}(U,\ZZ_2)_\mathfrak{m}\neq 0$. Hence $ S_{\gamma}(U,\FF_2)_\mathfrak{m}\neq 0$, and because reduction mod 2 is surjective, $S_\gamma(U,\ZZ_2)_\mathfrak{m}\neq 0$. Note also that $S_\gamma(U,\ZZ_2)$ is torsion-free. Thus there exists an automorphic representation $\pi'$ satisfying with the same properties as $\pi$ listed above, but furthermore at each $p\in \Sigma(D)$, $\pi'$ is the $\gamma_p$-twist of the Steinberg representation. This $\pi'$ is almost what we want, except that $\pi'$ might ramify at $q$. However, by Corollary \ref{Level raising form at auxiliary prime}, either $\pi'$ is unramified, or the quadratic twist $\pi'\otimes \chi_q$ is unramified at $q$, where $\chi_q$ is the unique quadratic character that ramifies only at $q$. By the choice of $q$, we are done as in Section \ref{sec:ordinarycaseproof}.
\begin{remark}\label{rem:allpositivesigns} The above argument shows also that if $\gamma$ is trivial then $S_{\gamma_\pi}(U,\ZZ_2)_\mathfrak{m}\neq 0$ implies $S_{\gamma}(U,\ZZ_2)_\mathfrak{m}\neq 0$ even if $U$ is not sufficiently small. It follows that there always exists a level raising form all whose signs are $+1$ in this case.
\end{remark}

\noindent\textbf{Case 2: $D$ is indefinite.} Let $V$ be the open (but not compact) subgroup of $G_D(\AA^\infty)$ such that 
\begin{itemize}
\item For $p|N$, $p\notin \Sigma(D)$, $V_p$ is $\Iw(p^{\ord_p(N)})$.
\item $V_q=\Iw_1(q)Z(\Z_q)$.
\item For $p\in \Sigma(D)$, $V_p=G_D(\QQ_p)=(D\otimes \QQ_p)^\times$.
\item For all other $p$, $V_p=\GL_2(\ZZ_p)$.
\end{itemize}  
\begin{lem} \label{elliptic elements} For any $g\in G_D(\AA^\infty)$, $gG_D(\QQ)g^{-1}\cap VZ(\AA^\infty)/Z(\QQ)$ has no non-trivial element of order $<q$.
\end{lem}
\begin{proof} Suppose $\gamma$ is a non-trivial element of order $h<q$. Writing the $q$-component of $\gamma$ as $kz$ with $z\in Z(\Q_q)$, $k\in \Iw_1(q)$, we have $k^hz^h$ is central, hence $k^h$ is central. But then $k^h$ is an element in $\ZZ_q^\times$ which is $1$ mod $q$, hence we can extract an $h$-th root $z'$ of it which is also $1$ mod $q$. Then $(kz'^{-1})^h=1$ but $kz'^{-1}\in \Iw_1(q)$ is an element of a pro-$q$ group, so $k=z'$. Hence $\gamma$ is central.
\end{proof}

Let $U\subset G_D(\AA^\infty)$ be the open compact subgroup of $V$ such that $U_p=V_p$ for all $p\notin \Sigma(D)$, and $U_p=(\cO_D\otimes\ZZ_p)^\times$ otherwise. As in the previous section, we define
\[X_V=G_D(\QQ)\setminus \HH^{\pm}\times G_D(\AA^\infty)/V\]
and similarly for $X_U$. Note that neither double coset will change if we replace $U$, $V$ by $UZ(\AA^\infty)$, $VZ(\AA^\infty)$, because $U_p\supset Z(\ZZ_p)$ and $Z(\AA^\infty)=Z(\QQ)\prod_p Z(\ZZ_p)$.
\begin{lem} $X_U$, $X_V$ are compact Riemann surfaces. The natural projection map $X_U\to X_V$ is unramified everywhere, and is a Galois covering with Galois group $$VZ(\AA^\infty)/UZ(A^\infty) \cong \prod_{p\in \Sigma(D)} \ZZ/2.$$
\end{lem}
\begin{proof} By strong approximation we have a finite decomposition $G_D(\AA^\infty)=\coprod G_D(\QQ) t_i V$. This gives
\[ X_V=\coprod  \Gamma_i \setminus\HH^\pm, \]
where $\Gamma_i= t_iVt_i^{-1} \cap G_D(\QQ)$ is a discrete group acting on $\HH^\pm$ through its infinite component (which is in $\GL_2(\RR)$). This gives $X_V$ the structure of a compact Riemann surface in the usual way.

The group $VZ(\AA^\infty)/UZ(A^\infty)$ acts on $X_U$ by right translation. Notice for each $p\in \Sigma(D)$, $V_pZ(\QQ_p)/U_pZ(\QQ_p)\cong \ZZ/2$. We claim that the action is faithful and free. Suppose $vz\in VZ(\AA^\infty)$ fixes a point represented by $(\tau,g)$ with $\tau\in \HH^\pm$, $g\in G_D(\AA)$. This means there exists $\gamma\in G_D(\QQ)$, $u\in U$ such that
\[(\tau,gvz)=(\gamma\tau,\gamma gu).\]
The element $\gamma\in G_D(\QQ)\cap gVZ(\AA^\infty)g^{-1}$ thus has a fixed point in $\HH^\pm$. Because $\gamma\in G_D(\QQ)$ acts discretely on $\HH^\pm$, it acts as a finite order automorphism on $\HH^\pm$, and there exists $h\leq 6$ such that $\gamma^h$ acts trivially on $\HH^\pm$. By Lemma \ref{elliptic elements} (and the fact $q$ is chosen to be large), $\gamma$ is central. But then $gvz=\gamma gu$ implies $vz\in UZ(\AA^\infty)$.
\end{proof}
Given the above lemmas, we proceed similar to the previous case. For $\gamma=(\gamma_p)_{p\in \Sigma(D)}$ a tuple of unramified characters $\gamma_p:\QQ_p^\times \to \mu_2$, we have the local system $\ZZ_2(\gamma)$ on $X_V$, given by twisting the trivial local system along the covering map $X_U\to X_V$. We have a short exact sequence
\[0\to H^1(X_V,\ZZ_2(\gamma))/2 \to H^1(X_V,\FF_2(\gamma))\to H^2(X_V,\ZZ_2(\gamma))[2].\]
The Hecke algebra $\TT=\mathbb{Z}[T_p,S_p]_{p\notin \Sigma(D)\cup \{q\}}$ acts on each term of the sequence, and the sequence is equivariant with respect to the $\mathbb{T}$-action. Since the $\TT$-action on the $H^2$ is Eisenstein, for $\mathfrak{m}$ the non-Eisenstein maximal ideal corresponding to the mod 2 representation $\bar\rho$, we have
\[ H^1(X_V,\ZZ_2(\gamma))_{\mathfrak{m}}\to  H^1(X_V,\F_2(\gamma))_{\mathfrak{m}}\]
is a surjection of Hecke modules. 
If $\gamma_\pi$ denote the tuple giving the signs of $\pi$, we have $H^1(X_V,\ZZ_2(\gamma_\pi))_{\mathfrak{m}}\neq 0$ and hence  $H^1(X_V,\F_2(\gamma_\pi))_{\mathfrak{m}}=H^1(X_V,\FF_2)_{\mathfrak{m}}\neq 0$. Because the Hecke action on $H^0$ is also Eisenstein, $H^1(X_V,\ZZ_2(\gamma))_\mathfrak{m}$ is torsion-free. Thus for any $\gamma$, $ H^1(X_V,\ZZ_2(\gamma))_{\mathfrak{m}}\neq 0$. The Jacquet--Langlands correspondence gives an automorphic representation $\pi'$ contributing to this space, and we are done by the same argument as in the definite case.
This finishes the proof of Theorem \ref{thm:levelraising} in the supersingular case.

\section{Preliminaries on local conditions}
\label{Preliminaries on local conditions}
So far we have only used items (1-4) of Assumption \ref{ass:main}. Henceforth we will assume  that all items in Assumption \ref{ass:main} holds for the elliptic curve $E$. Suppose $A$ is obtained from $E$ via level raising at $m\ge0$ primes (Definition \ref{def:levelraising}). Fix an isomorphism between $A[\lambda]\cong E[2]\otimes k$ and denote them by $V$. 

\begin{definition}
Let $v$ be a place of $\mathbb{Q}$. We define $H^1_\mathrm{ur}(\mathbb{Q}_v,V):=H^1(\mathbb{Q}_v^\mathrm{ur}/\mathbb{Q}_v, V^I)\subseteq H^1(\mathbb{Q}_v, V)$ consisting of classes which are split over an unramified extension of $\mathbb{Q}_v$. 
\end{definition}

\begin{definition}
Let $\mathcal{L}=\{\mathcal{L}_v\}$ be the collection of $k$-subspaces $\mathcal{L}_v\subseteq H^1(\mathbb{Q}_v, V)$, where $v$ runs over every place of $\mathbb{Q}$. We say $\mathcal{L}$ is a collection of \emph{local conditions} if $\mathcal{L}_v=H^1_\mathrm{ur}(\mathbb{Q}_v,V)$ for almost all $v$. We define the \emph{Selmer group} cut out by the local conditions $\mathcal{L}$ to be $$H^1_\mathcal{L}(V):=\{x\in H^1(\mathbb{Q} ,V): \res_v(x)\in \mathcal{L}_v, \text{for all } v\}.$$
\end{definition}

\begin{definition}
  We define $\mathcal{L}_v(A)$ to be the image of the local Kummer map $$A(\mathbb{Q}_v) \otimes_{\mathcal{O}_F} \mathcal{O}_F/\lambda \rightarrow  H^1(\mathbb{Q}_v, A[\lambda])=H^1(\mathbb{Q}_v,V).$$ The \emph{$\lambda$-Selmer group} of $A$ is defined to be the Selmer group cut out by $\mathcal{L}(A):=\{\mathcal{L}_v(A)\}$, denoted by $\Sel_\lambda(A/\mathbb{Q})$, or $\Sel(A)$ for short (if that causes no confusion). Its dimension as a $k$-space is called the \emph{$\lambda$-Selmer rank} of $A$, denoted by $\dim \Sel(A)$ for short.  For details on descent with endomorphisms, see the appendix of \cite{Gross2012}.
\end{definition}

\begin{definition}
The Weil pairing $E[2]\times E[2]\rightarrow \mu_2$ induces a perfect pairing  $V\times V\rightarrow k(1)$. We identify $V\cong V^*=\Hom(V, k(1))$ using this pairing. For each place $v$ of $\mathbb{Q}$, we define the cup product pairing  $$\langle\ ,\ \rangle_v: H^1(\mathbb{Q}_v, V)\times H^1(\mathbb{Q}_v, V)\rightarrow H^2(\mathbb{Q}_v, k(1))\cong k.$$ This is a perfect pairing by the local Tate duality.  We denote the annihilator of $\mathcal{L}_v$ by  $$\mathcal{L}_v^\perp:=\{x\in H^1(\mathbb{Q}_v, V): \langle x, y\rangle_v=0, \text{for all } y\in \mathcal{L}_v\}.$$ Then $\dim_k\mathcal{L}_v+\dim \mathcal{L}_v^\perp=\dim H^1(\mathbb{Q}_v,V)$ by the non-degeneracy of $\langle\ ,\ \rangle_v$. By the local Tate duality for the elliptic curve $E$, $\mathcal{L}_v(E)$ is equal to its own annihilator $\mathcal{L}_v(E)^\perp$ and hence $\dim \mathcal{L}_v(E)=\frac{1}{2}\dim H^1(\mathbb{Q}_v,V)$.
\end{definition}

\begin{lemma}\label{lem:dimension}
  Suppose $v\nmid 2N\infty$. Then $$\dim H^1(\mathbb{Q}_v,V)=2\dim H^1_\mathrm{ur}(\mathbb{Q}_v, V)=0,2,4,$$ if  $\Frob_v$ is of order $3, 2, 1$ acting on $V$ respectively.
\end{lemma}

\begin{proof}
The map $c\mapsto c(\Frob_v)$ induces an isomorphism $H^1_\mathrm{ur}(\mathbb{Q}_v,V)\cong V^I/(\Frob_v-1)V^I$, which has dimension $0,1,2$ if $\Frob_v$ has order $3,2,1$ respectively. If follows from  \cite[I.2.6]{Milne1986} that the annihilator of $H^1_\mathrm{ur}(\mathbb{Q}_v, V)$  is equal to itself, hence $\dim H^1(\mathbb{Q}_v,V)=2\dim H^1_\mathrm{ur}(\mathbb{Q}_v,V)$.
\end{proof}

Under our assumptions, the following lemma identifies the local conditions of the abelian variety $A$ purely in terms of the Galois representation $V$, which is the key to control Selmer ranks in level raising families in the next two sections.

\begin{lemma}\label{lem:localconditions}
  Suppose $A$ is obtained from $E$ via level raising at primes $q_1,\ldots q_m$ ($m\ge0$). Let $\mathcal{L}=\mathcal{L}(A)$ be the local conditions defining $\Sel(A)$. Then
  \begin{enumerate}
  \item For $v\nmid 2q_1\cdots q_m\infty$, $$\mathcal{L}_v=\mathcal{L}_v^\perp=H^1_\mathrm{ur}(\mathbb{Q}_v,V).$$
  \item For $v=\infty$, $$\mathcal{L}_v=H^1(\mathbb{Q}_v, V)=0.$$
  \item For $v=q_i$, if $\Frob_{q_i}$ has order 2 acting on $V$ and $A$ has sign $\varepsilon_i=+1$, then $H^1(\mathbb{Q}_v,V)$ is 2-dimensional and $$\mathcal{L}_v=\mathcal{L}_v^\perp=\im(H^1(\mathbb{Q}_v, W)\rightarrow H^1(\mathbb{Q}_v, V))$$ is 1-dimensional. Here $W$ is the unique $G_{\mathbb{Q}_v}$-stable line in $V$. Moreover, $\mathcal{L}_v$ and $H^1_\mathrm{ur}(\mathbb{Q}_v,V)$ are distinct lines.
  \item If $E$ is good at $v=2$, then $$\mathcal{L}_2=\mathcal{L}_2^\perp=H^1_\mathrm{fl}(\Spec \mathbb{Z}_2, \mathcal{E}[2])\otimes k,$$ where $\mathcal{E}/\mathbb{Z}_2$ be the Neron model of $E/\mathbb{Q}_2$ and $H^1_\mathrm{fl}(\Spec \mathbb{Z}_2, \mathcal{E}[2])$ is the flat cohomology group, viewed as a subspace of $H^1_\mathrm{fl}(\Spec \mathbb{Q}_2, E[2])=H^1(\mathbb{Q}_2, E[2])$.
  \item If $E$ is multiplicative at $v=2$, then $$\mathcal{L}_2=\mathcal{L}_2^\perp=\im(H^1(\mathbb{Q}_2, W)\rightarrow H^1(\mathbb{Q}_2, V)).$$ Here $W$ is the unique $G_{\mathbb{Q}_2}$-stable line in $V$.
  \end{enumerate}
\end{lemma} 

\begin{proof}
  (1) The fact that $\mathcal{L}_v=H^1_\mathrm{ur}(\mathbb{Q}_v,V)$ follows from \cite[Lemma 6]{Gross2012} and Remark \ref{rem:componentgroup}. 

  (2) By Remark \ref{rem:negativedisc}, the complex conjugation $c$ acts nontrivially on $V$, so $H^1(\mathbb{R}, V)=V^c/(1+c)V=0$.

  (3) Write $q=q_i$ and $\varepsilon=\varepsilon_i$ for short. Our argument closely follows the proof of \cite[Lemma 8]{Gross2012}. Let $\mathcal{A}/\mathbb{Z}_q$ be the Neron model of $A/\mathbb{Q}_q$. Let $\mathcal{A}^0/\mathbb{F}_q$ be the identity component of the special fiber of $\mathcal{A}$. Since $A$ is an isogeny factor of the new quotient of $J_0(Nq_1\cdots q_m)$, it has purely toric reduction at $q$: $\mathcal{A}^0/\mathbb{F}_q$ is a torus that is split over $\mathbb{F}_{q^2}$ and it is split over $\mathbb{F}_q$ if and only if $\varepsilon=+1$.  By the Neron mapping property, $\mathcal{O}_F$ acts on $\mathcal{A}^0$ and makes the character group $X^*(\mathcal{A}^0/\mathbb{F}_q) \otimes \mathbb{Q}$ a 1-dimensional $F$-vector space. Let $T/\mathbb{Q}_{q}$ be the split torus with character group $X^*(\mathcal{A}^0/\mathbb{F}_q)$. Then $\mathcal{O}_F$ naturally acts on $T$ (dual to the action on the character group).

By the theory of $q$-adic uniformization, we have a $G_{\mathbb{Q}_q}$-equivariant exact sequence $$0\rightarrow\Lambda\rightarrow T(\overline{\mathbb{Q}_q})\rightarrow A(\overline{\mathbb{Q}_q})\rightarrow0,$$ where $\Lambda$ is a free $\mathbb{Z}$-module with trivial $G_{\mathbb{Q}_q}$-action. Since $\mathcal{O}_F$ is a maximal order, $\Lambda$ is a locally free $\mathcal{O}_F$-module of rank one. Consider the following commutative diagram $$\xymatrix{ T(\mathbb{Q}_q) \otimes \mathcal{O}_F/\lambda \ar[r] \ar[d] & H^1(\mathbb{Q}_q, T[\lambda]) \ar[d] \\ A(\mathbb{Q}_q)  \otimes \mathcal{O}_F/\lambda \ar[r] &H^1(\mathbb{Q}_q, A[\lambda])}.$$ Here the horizontal arrows are the local Kummer maps and the vertical maps are induced by the $q$-adic uniformization. The left vertical map is surjective since its cokernel lies in $H^1(\mathbb{Q}_q,\Lambda)=\Hom(G_{\mathbb{Q}_q},\Lambda)$, which is zero as $\Lambda$ is torsion-free. The top horizontal map is also surjective since its cokernel maps into $H^1(\mathbb{Q}_q, T)$, which is zero by Hilbert 90 as $T$ is a split torus. It follows that $$\mathcal{L}_q=\im \left(H^1(\mathbb{Q}_q, T[\lambda])\rightarrow H^1(\mathbb{Q}_q,A[\lambda])\right).$$ Also, because $\Lambda$ has no $\lambda$-torsion, we see that $T[\lambda]\rightarrow A[\lambda]$ is a $G_{\mathbb{Q}_q}$-equivariant injection. But since $\Frob_q$ is assumed to have order 2 acting on  $V=A[\lambda]$, $V$ has a unique $G_{\mathbb{Q}_q}$-stable line $W$. Therefore $$\mathcal{L}_q=\im(H^1(\mathbb{Q}_q,W)\rightarrow H^1(\mathbb{Q}_q,V)).$$ 

It follows from Lemma \ref{lem:dimension} that $H^1(\mathbb{Q}_q,V)$ is 2-dimensional and $H^1_\mathrm{ur}(\mathbb{Q}_q,V)$ is 1-dimensional. A class $c\in H^1(\mathbb{Q}_q,W)=\Hom(G_{\mathbb{Q}_q},W)$ is determined by its image on $\sigma$ (a lift of $\Frob_v$) and a tame generator $\tau$. Suppose $E[2](\mathbb{Q}_q)=\langle P\rangle$, then the class $c(\tau)=0$, $c(\sigma)=P$ is cohomologous to zero in $H^1(\mathbb{Q}_v,V)$ (equal to the coboundary of a non $\mathbb{Q}_q$-rational point in $E[2]$). We see that $\mathcal{L}_q$ is generated by the class $c(\tau)= P$, $c(\sigma)=0$. So $\mathcal{L}_q$ is 1-dimensional and $\mathcal{L}_q\cap H^1_\mathrm{ur}(\mathbb{Q}_q,V)=0$. This finishes the proof.
  
  (4) Let $\mathcal{E}/\mathbb{Z}_2$ be the Neron model of $E/\mathbb{Q}_2$ and $\mathcal{A}/\mathbb{Z}_2$ be the Neron model of $A/\mathbb{Q}_2$. We claim that $\mathcal{E}[2]\otimes k=\mathcal{A}[\lambda]$ over $\mathbb{Z}_2$ (extending the isomorphism $E[2] \otimes k=A[\lambda]$).

  First consider the case that $E$ is supersingular at 2.  Let $W$ be the strict henselization of $\mathbb{Z}_2$. Let $F$ be the fraction field of $W$ and $I$ be the absolute Galois group of $F$ (i.e., the inertia subgroup at 2). Notice $E[2]$ is an irreducible $\mathbb{F}_2[I]$-module (\cite[p.275, Prop.12]{Serre1972}, see also \cite[Theorem 1.1]{Conrad1997}), hence by \cite[3.3.2.3$^\circ$]{Raynaud1974}, we know that $E[2]$ has a unique finite flat model over $W$. Since the descent datum from $W$ to $\mathbb{Z}_2$ is determined by that of the generic fiber, $E[2]$ has a \emph{unique} finite flat model over $\mathbb{Z}_2$ as well. Now $E[2] \otimes k$ is a direct sum of $[k:\mathbb{F}_2]$ copies of $E[2]$, by the standard 5-lemma argument (\cite[Prop 4.2.1]{Tate1997}), we know that $E[2] \otimes k$ also has a unique finite flat model over $\mathbb{Z}_2$. We conclude that this unique finite flat model of $E[2]\otimes k$ must be isomorphic to $\mathcal{E}[2] \otimes k=\mathcal{A}[\lambda]$.

  Now consider the case that $E$ is ordinary at 2. Then $\mathcal{E}[2]$ is an extension of $\mathbb{Z}/2 \mathbb{Z}$ by $\mu_2$ over $\mathbb{Z}_2$. Notice $b_2\equiv a_2\not\equiv0 \pmod{\lambda}$ by construction, we know that $\mathcal{A}[\lambda]$ is also ordinary, i.e., an extension of $\mathbb{Z}/2 \mathbb{Z} \otimes k$ by $\mu_2 \otimes k$ over $\mathbb{Z}_2$. To show that  $\mathcal{E}[2]\otimes k=\mathcal{A}[\lambda]$, it suffices to show that $E[2] \otimes k= A[\lambda]$ has a unique finite flat model over $\mathbb{Z}_2$ that is an extension of  $\mathbb{Z}/2 \mathbb{Z} \otimes k$ by $\mu_2 \otimes k$. This is true because of Assumption \ref{ass:main} (4) that $G_{\mathbb{Q}_2}$ acts nontrivially on $E[2]$. In fact, the generic fiber map $$\Ext_{\mathbb{Z}_2}(\mathbb{Z}/2 \mathbb{Z},\mu_2)\rightarrow\Ext_{\mathbb{Q}_2}(\mathbb{Z}/2 \mathbb{Z},\mu_2)$$ between the extension groups in the category of fppf sheaves of $\mathbb{Z}/2 \mathbb{Z}$-modules can be identified with the natural map $$H^1_\mathrm{fppf}(\mathbb{Z}_2,\mu_2)\cong \mathbb{Z}_2^\times/(\mathbb{Z}_2^\times)^2\rightarrow H^1_\mathrm{fppf}(\mathbb{Q}_2,\mu_2)\cong \mathbb{Q}_2^\times/(\mathbb{Q}_2^\times)^2.$$ This map is injective. As a direct sum of $[k:\mathbb{F}_2]^2$ copies of this map, it follows that $$\Ext_{\mathbb{Z}_2}(\mathbb{Z}/2 \mathbb{Z} \otimes k, \mu_2 \otimes k)\rightarrow\Ext_{\mathbb{Q}_2}(\mathbb{Z}/2 \mathbb{Z} \otimes k, \mu_2 \otimes k)$$ is also injective, which means that the extension class of such a finite flat model $\mathcal{V}$ of $E[2] \otimes k$ is determined by the extension class of the generic fiber of $\mathcal{V}$. But  $G_{\mathbb{Q}_2}$ acts nontrivially on $E[2]$, there is a unique $\mathbb{F}_2$-subspace of dimension $[k:\mathbb{F}_2]$ in $E[2] \otimes k$ with trivial $G_{\mathbb{Q}_2}$-action, so the extension class of the generic fiber of $\mathcal{V}$ is uniquely determined by $E[2] \otimes k$, as desired.

  In both cases, we have $\mathcal{E}[2]\otimes k= \mathcal{A}[\lambda]$. Now by \cite[Lemma 7]{Gross2012}, we know that $$\mathcal{L}_v=H^1_\mathrm{fl}(\mathbb{Z}_2, \mathcal{A}[\lambda])=H^1_\mathrm{fl}(\mathbb{Z}_2, \mathcal{E}[2]) \otimes k.$$ 

  (5) By Assumption \ref{ass:main} (3), there exists a unique $G_{\mathbb{Q}_2}$-stable line $W$ in $V$. If $A$ has split toric reduction at 2, the claim follows from the same argument as in (3) using the 2-adic uniformization of $A/\mathbb{Q}_2$. Now let us assume that $A$ has non-split toric reduction. Since $G_{\mathbb{Q}_2}$ acts on $V$ nontrivially and the image of $\bar\rho|_{\mathbb{Q}_2}$ has order 2, one easily sees that $\dim H^1(\mathbb{Q}_2, V)=4$ by the Euler characteristic formula and $$\dim \im(H^1(\mathbb{Q}_2, W)\rightarrow H^1(\mathbb{Q}_2,V))=2$$ by the long exact sequence in Galois cohomology associated to the short exact sequence $$0\rightarrow W\rightarrow V\rightarrow W/V\rightarrow0.$$ Since $\mathcal{L}_2$ is a maximal isotropic subspace of $H^1(\mathbb{Q}_2,V)$ by the local Tate duality for $A$, we know that $\dim\mathcal{L}_2=2$, half of the dimension of $H^1(\mathbb{Q}_2, V)$. To prove the claim, it suffices to show that $\mathcal{L}_2$ contains $\im(H^1(\mathbb{Q}_2, W)\rightarrow H^1(\mathbb{Q}_2,V))$.

Let $T$ be the split torus over $\mathbb{Q}_2$ with character group $X^*(\mathcal{A}^0/\mathbb{F}_2)$. Let $\chi$ be the unramified quadratic character $\chi: \Gal(\mathbb{Q}_4/\mathbb{Q}_2)\rightarrow\{\pm1\}$ and $T(\chi)$ be the $\chi$-twist of $T$. We have a $G_{\mathbb{Q}_2}$-equivariant exact sequence $$0\rightarrow\Lambda(\chi)\rightarrow T(\chi)(\overline{\mathbb{Q}_2})\rightarrow A(\overline{\mathbb{Q}_2})\rightarrow0,$$ where $\Lambda$ is a locally free $\mathcal{O}_F$-module of rank one with trivial $G_{\mathbb{Q}_2}$-action. As in (3), consider the following commutative diagram $$\xymatrix{ T(\chi)(\mathbb{Q}_2) \otimes \mathcal{O}_F/\lambda \ar[r] \ar[d] & H^1(\mathbb{Q}_2, T(\chi)[\lambda]) \ar[d] \\ A(\mathbb{Q}_2)  \otimes \mathcal{O}_F/\lambda \ar[r] &H^1(\mathbb{Q}_2, A[\lambda])}.$$ Since the image of the right vertical arrow is $\im(H^1(\mathbb{Q}_2,W)\rightarrow H^1(\mathbb{Q}_2,V))$, we are done if the left vertical arrow is surjective, or equivalently,
\begin{equation}
  \label{eq:nonsplit}
  \ker\left(H^1(\mathbb{Q}_2, \Lambda(\chi))_\lambda\rightarrow H^1(\mathbb{Q}_2, T(\chi))_\lambda\right) \otimes \mathcal{O}_F/\lambda
\end{equation} is zero. Since $H^1(\mathbb{Q}_4, \Lambda(\chi))=0$ ($\Lambda$ is torsion-free) and $H^1(\mathbb{Q}_4, T(\chi))=0$ by Hilbert 90 ($T(\chi)$ splits over $\mathbb{Q}_4$), by inflation-restriction we know that $$H^1(\mathbb{Q}_2,\Lambda(\chi))=H^1(\mathbb{Q}_4/\mathbb{Q}_2, \Lambda(\chi))=\Lambda/2\Lambda,$$ and $$H^1(\mathbb{Q}_2,T(\chi))=H^1(\mathbb{Q}_4/\mathbb{Q}_2, T(\chi)(\mathbb{Q}_4))=T(\mathbb{Q}_2)/\mathbb{N}(T(\mathbb{Q}_4)),$$ where $\mathbb{N}: T(\mathbb{Q}_4)\rightarrow T(\mathbb{Q}_2)$ is the norm map. The domain and target in (\ref{eq:nonsplit})  are finite $\mathcal{O}_{F,\lambda}$-modules of the same size because $\Lambda$ is a locally free $\mathcal{O}_F$-module of rank one. Hence it suffices to show that $$H^1(\mathbb{Q}_2, \Lambda(\chi))_\lambda\rightarrow H^1(\mathbb{Q}_2, T(\chi))_\lambda$$ is surjective, which can be checked after tensoring with $\mathcal{O}_F/\lambda$, i.e., $$\Lambda/\lambda\Lambda\rightarrow T(\mathbb{Q}_2)/\mathbb{N}(T(\mathbb{Q}_4)) \otimes \mathcal{O}_F/\lambda$$ is surjective. Since these are 1-dimensional $k$-vector spaces, it suffices to show this last map is nonzero. We claim that for any $a\in \Lambda-\lambda \Lambda$, we have $a\not\in \mathbb{N}(T(\mathbb{Q}_4))$. This is true because of Assumption \ref{ass:main} (3) that $\bar\rho$ is ramified at 2. In fact, let $\lambda^{-1}{\Lambda}=\{t \in T(\chi): \lambda t\subseteq \Lambda\}$,  then $A[\lambda]\cong \lambda^{-1}\Lambda/\Lambda$ (notice that $\lambda^{-1}\Lambda/\Lambda$ is 2-dimensional over $k$: the torsion subgroup $T(\chi)[\lambda]$ gives a $k$-line in $\lambda^{-1}\Lambda/\Lambda$, whose quotient is isomorphic to $\Lambda/\lambda \Lambda$). We know that $\lambda^{-1}(a)$ generates a ramified extension of $\mathbb{Q}_2$. On the other hand, for any $b\in T(\mathbb{Q}_4)$,  $\lambda^{-1}(\mathbb{N}(b))=\lambda'(\sqrt{\mathbb{N}(b)})$,  where $\lambda'$ is an integral ideal of $\mathcal{O}$ such that  $\lambda\lambda'=(2)$. Since $\mathbb{Q}_2(\sqrt{\mathbb{N}(b)})/\mathbb{Q}_2$ is unramified, we know that $\lambda^{-1}(\mathbb{N}(b))$ generates an unramified extension of $\mathbb{Q}_2$. Therefore $a$ is not of the form $\mathbb{N}(b)$, as desired.

\bigskip 
\noindent Finally, in cases (1-2), (4-5), we have $\mathcal{L}_v=\mathcal{L}_v(E)$  and the claim $\mathcal{L}_v=\mathcal{L}_v^\perp$ follows from the local Tate duality for $E$. In case (3), the claim $\mathcal{L}_v=\mathcal{L}_v^\perp$ is clear since $H^1(\mathbb{Q}_v,V)$ is 2-dimensional.
\end{proof}

\begin{remark}\label{rem:localat2}
When Assumption \ref{ass:main} (4) is not satisfied, it is possible that $\mathcal{L}_2(E)\ne\mathcal{L}_2(A)$ (see Remark \ref{rem:trivialat2}).
\end{remark}

\section{Rank lowering}
\label{Rank lowering}
\begin{lemma}\label{lem:lowering}
  Suppose $\mathcal{L}$ and $\mathcal{L}'$ are two collections of local conditions. Let $w$ be a place of $\mathbb{Q}$.
  \begin{enumerate}
  \item 
    Assume that $\mathcal{L}_v=\mathcal{L}_v'=\mathcal{L}_v^\perp$ for all $v\ne w$.
    Then $\dim H^1_\mathcal{L} (V)$ and $\dim
    H^1_\mathcal{L'}(V)$ differ by at most $\frac{1}{2}\dim
    H^1(\mathbb{Q}_{w}, V)$.
  \item If we further assume that
    \begin{enumerate}
    \item $H^1(\mathbb{Q}_{w}, V)$ is 2-dimensional,
    \item $\mathcal{L}_{w}$, $\mathcal{L}_{w}'$ are distinct lines,
    \item $\res_{w}(H^1_\mathcal{L}(V))\ne0$.
    \end{enumerate}
    Then we have $$\dim H^1_\mathcal{L'}(V)=\dim H^1_\mathcal{L}(V)-1.$$
  \end{enumerate}
\end{lemma}
  
\begin{proof}
  (1) Define the strict local conditions $\mathcal{S}$ by $\mathcal{S}_v=\mathcal{L}_v$ for $v\ne w$ and $\mathcal{S}_{w}=0$. Similarly, define the relaxed local conditions $\mathcal{R}$ by $\mathcal{R}_v=\mathcal{L}_v$ for $v\ne w$ and $\mathcal{R}_w=H^1(\mathbb{Q}_w, V)$. Then we have $$H^1_\mathcal{S}(V)\subseteq H^1_\mathcal{L}(V) \subseteq H^1_\mathcal{R}(V),\quad H^1_\mathcal{S}(V)\subseteq H^1_\mathcal{L'}(V) \subseteq H^1_\mathcal{R}(V).$$ The assumptions implies that $\mathcal{R}^\perp=\mathcal{S}$. By \cite[Theorem 2.18]{Darmon1997}, we can compare the dual Selmer groups: $$\frac{\#H^1_\mathcal{S}(V)}{\# H^1_\mathcal{R}(V)}=\prod_v\frac{\# \mathcal{S}_v}{\#H^0(\mathbb{Q}_v, V)}, \quad \frac{\#H^1_\mathcal{R}(V)}{\# H^1_\mathcal{S}(V)}=\prod_v\frac{\# \mathcal{R}_v}{\#H^0(\mathbb{Q}_v, V)}. $$ It follows that $$\dim H^1_\mathcal{R}(V)-\dim H^1_\mathcal{S}(V)=\frac{1}{2}(\dim \mathcal{R}_w-\dim \mathcal{S}_w)=\frac{1}{2}\dim H^1(\mathbb{Q}_w,V).$$ So the first claim is proved.

  (2) Let $c_1,c_2\in H^1_\mathcal{L}(V)$, then $$\sum_v \langle \res_v(c_1),\res_v(c_2)\rangle_v=0$$ by global class field theory. The assumption $\mathcal{L}_v=\mathcal{L}_v^\perp$ implies that $$\langle \res_v(c_1),\res_v(c_2)\rangle_v=0,\quad v\ne w.$$ Hence $\langle \res_w(c_1),\res_w(c_2)\rangle_w=0$ as well. It follows that $\res_w(H^1_\mathcal{L}(V))$ is a totally isotropic subspace of $H^1(\mathbb{Q}_w, V)$ for the pairing $\langle\ ,\ \rangle_w$. The same argument shows that $\res_w(H^1_\mathcal{L'}(V))$ and $\res_w(H^1_\mathcal{R}(V))$ are also totally isotropic subspaces of $H^1(\mathbb{Q}_w, V)$. The isotropic subspaces are isotropic lines or zero by (a). Now (c) implies that $\res_w(H^1_\mathcal{L}(V))$ must be the line $\mathcal{L}_w\subseteq H^1(\mathbb{Q}_w, V)$. Thus $\res_w(H^1_\mathcal{R}(V))$ must also be $\mathcal{L}_w$, as it contains $\res_w(H^1_\mathcal{L}(V))$. We thus know that $H^1_\mathcal{L}(V)=H^1_\mathcal{R}(V)$. Notice that $$\res_w(H^1_\mathcal{L'}(V))\subseteq \mathcal{L}'_w\cap \res_w(H^1_\mathcal{R}(V))=\mathcal{L}'_w\cap \mathcal{L}_w,$$ which is zero by (b), we know that $H^1_\mathcal{L'}(V)=H^1_\mathcal{S}(V)$.  The first part tells us that $$\dim H^1_\mathcal{R}(V)-\dim H^1_\mathcal{S}(V)=1.$$ So the desired result is proved.
\end{proof}

\begin{cor}\label{cor:difference}
  Suppose $A$ is obtained from $E$ via level raising at one prime $q$. Then $\dim \Sel(A)$ and $\dim \Sel(E)$ differ by at most 1 (resp. 2) when $\Frob_q$ is of order 2 (resp. 1)  acting on $V$.
\end{cor}

\begin{proof}
  This follows immediately from Lemma \ref{lem:lowering} (1), Lemma \ref{lem:localconditions} and Lemma \ref{lem:dimension}.
\end{proof}

\begin{remark}\label{rem:trivialat2}
  The conclusion of Corollary \ref{cor:difference} may fail when Assumption \ref{ass:main} (4) is not satisfied due to the uncertainty of the local conditions at 2. For example, the elliptic curve $E=2351a1: y^2+xy+y=x^3-5x-5$ has trivial $\bar\rho|_{G_{\mathbb{Q}_2}}$. The elliptic curve $A=25861i1: y^2+xy+y=x^3+x^2-17x+30 $ is obtained from $E$ via level raising at $q=11$. One can compute that $\Frob_q$ has order 2 but $\dim\Sel(E)=0$ and $\dim\Sel(A)=2$ differ by 2.
\end{remark}

Recall that $L=\mathbb{Q}(E[2])$. The inflation restriction exact sequence gives us $$0\rightarrow H^1(L/\mathbb{Q}, V)\rightarrow H^1(\mathbb{Q}, V)\rightarrow H^1(L, V)^{\Gal(L/\mathbb{Q})}\rightarrow H^2(L/\mathbb{Q}, V).$$ Since $V$ is the irreducible 2-dimensional representation of $\Gal(L/\mathbb{Q})\cong S_3$, we have $H^1(L/\mathbb{Q}, V)=H^2(L/\mathbb{Q}, V)=0$. Since $G_L$ acts trivially on $V$, we know that $H^1(L, V)=\Hom(G_L, V)$. Therefore we obtain an isomorphism $$H^1(\mathbb{Q}, V)\cong \Hom(G_L, V)^{S_3}.$$  This allows us to view $c\in H^1(\mathbb{Q}, V)$ as a homomorphism $f: G_L\rightarrow V$ that is equivariant under the $S_3$-action. Namely, for any $g\in G_\mathbb{Q}$, $h\in G_L$, we have $$f(ghg^{-1})=\bar g. f(h),$$ where $\bar g$ is the image of $g$ in $\Gal(L/\mathbb{Q})\cong S_3$. 

\begin{lemma}\label{lem:surjectivity}
  Let $c_1,\ldots,c_r\in H^1(\mathbb{Q},V)$ be linearly independent elements. Let $f_1,\ldots, f_r:G_L\rightarrow V$ be the corresponding homomorphisms. Then the homomorphism $$f: G_L \rightarrow V^r, \quad g\mapsto (f_1(g),\ldots,f_r(g))$$ is surjective.
\end{lemma}

\begin{proof}
  Since $f_i$ is $S_3$-equivariant, we know that the image of $f$ is a $S_3$-subrepresentation of $V^r$, hence must be isomorphic to $V^s$ for some $s\le r$. Therefore $$f_i\in \Hom( G_L/\ker f, V)^{S_3}\cong \Hom(V^s, V)^{S_3}$$ lies in a $s$-dimensional space. But since $\{c_i\}$ are linearly independent, the homomorphisms $\{f_i\}$ are also linearly independent, we know that $s\ge r$. The surjectivity follows.
\end{proof}

\begin{lemma}\label{lem:loweringdensity}
Suppose $A$ is obtained from $E$ via level raising. Suppose $\Sel(A)\ne0$. Then there exists a positive density set of primes $w$ satisfying the following.
  \begin{enumerate}
  \item $H^1(\mathbb{Q}_w,V)$ is 2-dimensional.
  \item $\res_w(\Sel(A))\ne0$,
  \end{enumerate}  
\end{lemma}

\begin{proof}
  By Lemma \ref{lem:dimension}, the first condition that $H^1(\mathbb{Q}_w,V)$ is 2-dimensional is equivalent to that $\Frob_w\in\Gal(L/\mathbb{Q})\cong S_3$ has order 2.

  Let $c\in \Sel(A)\subseteq H^1(\mathbb{Q}, V)$ be a nonzero class. Let $f:G_L\rightarrow V$ be the corresponding homomorphism. We claim that there exists $g\in G_\mathbb{Q}$ such that $\bar g$ has order 2 and $f(g^2)\ne0$. Take any transposition in $S_3$ and lift it to some $g\in G$. We are done if $f(g^2)\ne0$. Otherwise, since $V$ is a 2-dimensional irreducible representation of $S_3$, we know there exists $v\in V$ such that $\bar g. v+ v\ne0$. By Lemma \ref{lem:surjectivity}, we can choose $h\in G_L$ such that $f(h)= v$. Let $g'=gh\in G_\mathbb{Q}$. Then $\bar g'$ has order 2 and $$f(g'^2)=f(ghgh)=f(ghg^{-1}\cdot g^2\cdot h)=\bar g.f(h)+f(g^2)+f(h)=\bar g. f(h)+f(h).$$  Therefore $f(g'^2)=\bar g. v+v\ne0$ and the claim is proved.

  It follows form the previous claim and the Chebotarev density theorem that there exists a positive density set of primes $w$, such that $\Frob_w$ has order 2 in $\Gal(L/\mathbb{Q})$ and $f(\Frob_w^2)\ne0$. Let $u$ be a prime of $L$ over $w$. Since $$H^1(L_u/\mathbb{Q}_w, V)=H^2(L_u/\mathbb{Q}_w, V)=0,$$ we know that $$\res_w: H^1(\mathbb{Q}, V)\rightarrow H^1(\mathbb{Q}_w,V)$$ can be identified as $$\Hom(G_L, V)^{S_3}\rightarrow \Hom(G_{L_u}, V)^{\Gal(L_u/\mathbb{Q}_w)},\quad f\mapsto f|_{G_{L_u}}$$ by restricting $f$ to the decomposition group $G_{L_u}$. Therefore $\res_w(c)=f|_{G_{L_u}}\ne0$, as $f(\Frob_w^2)\ne0$. This completes the proof. 
\end{proof}

\begin{proposition}\label{pro:ranklowering}
  Suppose $A$ is obtained from $E$ via level raising at primes $q_1,\ldots,q_m$ ($m\ge0$) such that for any $i\le m$,
  \begin{enumerate}
  \item $H^1(\mathbb{Q}_{q_i},V)$ is 2-dimensional.
  \item $\varepsilon_{i}=\varepsilon_i(A)=+1$, and
  \item $\dim\Sel(A)\ge 1$.
  \end{enumerate}
Then there exists a positive density set of primes $q_{m+1}$ and $A'$ obtained from $E$ via level raising at primes $q_1,\ldots,q_m, q_{m+1}$ such that
  \begin{enumerate}
  \item $H^1(\mathbb{Q}_{q_{m+1}},V)$ is 2-dimensional.
  \item $\varepsilon_i'=\varepsilon_i(A')=+1$ for any $i\le m+1$, and  
  \item $\dim\Sel(A')=\dim\Sel(A)-1$.
  \end{enumerate}
\end{proposition}

\begin{proof}
Lemma \ref{lem:loweringdensity} ensures the existence of a positive density set of primes $w=q_{m+1}$ such that $H^1(\mathbb{Q}_w,V)$ is 2-dimensional and $\res_w(\Sel(A))\ne0$. For such $w$, we choose $A'$ using Theorem \ref{thm:levelraising} with the prescribed signs $\varepsilon_i=+1$ ($i\le m+1$). Then the local conditions $\mathcal{L}=\mathcal{L}(A)$ and $\mathcal{L}'=\mathcal{L}(A')$ satisfies $\mathcal{L}_v=\mathcal{L}_v'=\mathcal{L}_v^\perp$ for $v\ne q_{m+1}$ by Lemma \ref{lem:localconditions}. For $w=q_{m+1}$, $\mathcal{L}_w$ and $\mathcal{L}_w'$ are distinct lines by Lemma \ref{lem:localconditions} as well. Now we can apply Lemma \ref{lem:lowering} to conclude that $\dim\Sel(A')=\dim\Sel(A)-1$.
\end{proof}

\begin{theorem}\label{thm:ranklowering}
    Suppose $E/\mathbb{Q}$ satisfies Assumption \ref{ass:main}. Then for any given integer $0\le n<\dim\Sel(E)$, there exists infinitely many abelian varieties $A/\mathbb{Q}$ obtained from $E/\mathbb{Q}$ via level raising, such that $$\dim\Sel(A)=n.$$ 
\end{theorem}

\begin{proof}
  It follows immediately from Proposition \ref{pro:ranklowering} by induction on the number of level raising primes $m$.
\end{proof}

\section{Rank raising}
\label{Rank raising}
To raise the rank, we need more refined control over the local conditions. For this purpose, we not only need the bilinear pairing $\langle\ , \rangle_v$, but also a quadratic form $Q_v$ giving rise to it. To define $Q_v$, first recall that the line bundle $\mathcal{L}=\mathcal{O}_E(2\infty)$ on $E$ induces a degree 2 map $$E\rightarrow \mathbb{P}^1=\mathbb{P}(H^0(E,\mathcal{L})).$$ For $P\in E$, let $\tau_P$ be the translation by $P$ on $E$. Since for $P\in E[2]$, $\tau_P^* \mathcal{L}\cong \mathcal{L}$, the translation by $E[2]$ induces an action of $E[2]$ on $\mathbb{P}^1$, i.e., a homomorphism $E[2]\rightarrow \PGL_2$. The short exact sequence $$0\rightarrow \mathbb{G}_m\rightarrow \GL_2\rightarrow \PGL_2\rightarrow 0$$ induces the connecting homomorphism in nonabelian Galois cohomology $$H^1(\mathbb{Q}, \PGL_2)\rightarrow H^2(\mathbb{Q},\mathbb{G}_m).$$

\begin{definition}
We define $Q$ to be the composition $$Q: H^1(\mathbb{Q}, E[2])\rightarrow H^1(\mathbb{Q}, \PGL_2)\rightarrow H^2(\mathbb{Q}, \mathbb{G}_m).$$For a place $v$ of $\mathbb{Q}$, we denote its restriction by $$Q_v: H^1(\mathbb{Q}_v, E[2])\rightarrow H^1(\mathbb{Q}_v, \PGL_2)\rightarrow H^2(\mathbb{Q}_v, \mathbb{G}_m).$$ By local class field theory, $H^2(\mathbb{Q}_v, \mathbb{G}_m)\cong \mathbb{Q}/\mathbb{Z}$ and so $Q_v$ takes value in $H^2(\mathbb{Q}_v, \mathbb{G}_m)[2]\cong\mathbb{Z}/2 \mathbb{Z}$. By \cite[\S 4]{ONeil2002}, $Q_v$ is an quadratic form and extending scalars we obtain a quadratic form $$Q_v: H^1(\mathbb{Q}_v, V)\rightarrow k,$$ whose associated bilinear form is given by $\langle\ , \ \rangle_v$.  
\end{definition}

\begin{definition}
We say a subspace $W\subseteq H^1(\mathbb{Q}_v, V)$ is \emph{totally isotropic for $Q_v$} if $Q_v|_W=0$. We say $W$ is \emph{maximal totally isotropic} if it is totally isotropic and $W=W^\perp$.
\end{definition}

\begin{remark}\label{rem:poonenrains}
The local condition $\mathcal{L}_v=\mathcal{L}_v(E)$ is maximal totally isotropic for $Q_v$ by \cite[Prop. 4.11]{Poonen2012} (this is also implicit in \cite[Prop. 2.3]{ONeil2002}).
\end{remark}

\begin{remark}\label{rem:isotropic}
 As $\mathrm{char}(k)=2$, the requirement $Q_v|_W=0$ is stronger than $\langle\ ,\ \rangle_v|_W=0$. For example, if $\dim H^1(\mathbb{Q}_v,V)=2$, then all three lines in $H^1(\mathbb{Q}_v,V)$ are isotropic for $\langle\ ,\ \rangle_v$, but only two of them are isotropic for $Q_v$ (since $(H^1(\mathbb{Q}_v,V), Q_v)$ is isomorphic to $(k^2, xy)$ as quadratic spaces).
\end{remark}

We replace the role of bilinear form $\langle\ ,\  \rangle_v$ by the quadratic form $Q_v$ and obtain the following more refined result analogous to Lemma \ref{lem:lowering}.

\begin{lemma}\label{lem:rankraising}
  Suppose $\mathcal{L}$ and $\mathcal{L}'$ are two collections of local conditions. Let $w$ be a place of $\mathbb{Q}$. Assume that
  \begin{enumerate}
  \item  $\mathcal{L}_v=\mathcal{L}_v'$ are maximal totally isotropic for $Q_v$ (for any $v\ne w$),
  \item $H^1(\mathbb{Q}_{w}, V)$ is 2-dimensional,
  \item $\mathcal{L}_{w}$, $\mathcal{L}_{w}'$ are distinct lines and are both isotropic for $Q_w$.
  \end{enumerate}
  Then $$\dim H^1_\mathcal{L'}(V)=\dim H^1_\mathcal{L}(V)\pm1.$$
  Moreover, $\res_w( H^1_\mathcal{L}(V))=0$ if and only if $$\dim H^1_\mathcal{L'}(V)=\dim H^1_\mathcal{L}(V)+1.$$
\end{lemma}

\begin{proof}
  By the proof of Lemma \ref{lem:lowering} (1), we obtain that $$H^1_\mathcal{S}(V)\subseteq H^1_\mathcal{L}(V) \subseteq H^1_\mathcal{R}(V),\quad H^1_\mathcal{S}(V)\subseteq H^1_\mathcal{L'}(V) \subseteq H^1_\mathcal{R}(V)$$ and $$\dim H^1_\mathcal{R}(V)=\dim H^1_\mathcal{S}(V)+1,$$ since $\dim H^1(\mathbb{Q}_w, V)=2$. By global class field theory, for any class $c\in H^1(\mathbb{Q},V)$, we have $$\sum_v Q_v(\res_v(c))=0.$$ The assumption that $\mathcal{L}_v$ is totally isotropic for $Q_v$ (for any $v\ne w$)  implies that $Q_w(\res_w(c))=0$ for any $c\in H^1_\mathcal{L}(V)$. In other words, the image $\res_w(H^1_\mathcal{L}(V))$ is a totally isotropic subspace for $Q_w$. Similarly, the image of $H^1_{\mathcal{L}'}(V)$, $H^1_\mathcal{R}(V)$ under $\res_w$ are all totally isotropic subspaces for $Q_w$. Since $H^1_\mathcal{R}(V)\ne H^1_\mathcal{S}(V)$ and $H^1(\mathbb{Q}_w, V)$ is 2-dimensional, we know that $\res_w(H^1_\mathcal{R}(V))$ must be an isotropic line for $Q_w$. But there are exactly two isotropic lines for $Q_w$ (see Remark \ref{rem:isotropic}), which must be $\mathcal{L}_w$ and $\mathcal{L}_w'$ since they are assumed to be distinct. When $\res_w(H^1_\mathcal{L}(V))=\mathcal{L}_w$, it follows that $\res_w(H^1_\mathcal{R}(V))=\mathcal{L}_w$ and  $$H^1_\mathcal{R}(V)=H^1_\mathcal{L}(V),\quad H^1_\mathcal{S}(V)=H^1_{\mathcal{L}'}(V).$$ When $\res(H^1_\mathcal{L}(V))=0$, it follows that $\res_w(H^1_\mathcal{R}(V))=\mathcal{L}_w'$ and $$H^1_\mathcal{R}(V)=H^1_{\mathcal{L}'}(V),\quad H^1_\mathcal{S}(V)=H^1_\mathcal{L}(V).$$ This finishes the proof.
\end{proof}

\begin{lemma}\label{lem:isotropic}
 Suppose $w\nmid 2N\infty$ is a prime such that $H^1(\mathbb{Q}_w,V)$ is 2-dimensional. Let $\mathcal{L}_w'=\im(H^1(\mathbb{Q}_w, W)\rightarrow H^1(\mathbb{Q}_w, V))$, where $W$ is the unique $G_{\mathbb{Q}_w}$-stable line in $V$. If $\Frob_w^2$ is sufficiently close 1 (depending only on $E$), then $\mathcal{L}_w'$ is an isotropic line for $Q_w$.
\end{lemma}

\begin{proof}
  By the proof of Lemma \ref{lem:localconditions} (3), we know that the line $\mathcal{L}_w'$ is generator by the class represented by the cocycle $c(\sigma)=0$, $c(\tau)=P$, where $\sigma$ is a lift of of $\Frob_w$, $\tau$ is a generator of the tame quotient $\Gal(\mathbb{Q}_w^t/\mathbb{Q}_w^\mathrm{ur})$  and $P$ is a generator of $E[2](\mathbb{Q}_w)$.

  We provide an explicit way to compute its image under $Q_w: H^1(\mathbb{Q}_w, E[2])\rightarrow H^1(\mathbb{Q}_w, \PGL_2)$. Recall that $H^1(\mathbb{Q}_w, \PGL_2)$ classifies forms of $\mathbb{P}^1$, i.e., algebraic varieties $S/\mathbb{Q}_w$ which become isomorphic to $\mathbb{P}^1$ over $\overline{\mathbb{Q}}_w$. For any cocycle $c$, the corresponding form $S$ can be described as follows. As a set, $S=\mathbb{P}^1(\overline{\mathbb{Q}}_w)$. The Galois action of $g\in G_{\mathbb{Q}_w}$ on $x\in S$ is given by $g. x=c(g).g(x)$. The cocycle $c$ is the trivial class in $H^1(\mathbb{Q}_w, \PGL_2)$ if and only if $S(\mathbb{Q}_w)\ne\varnothing$.

  For our specific cocycle $c(\sigma)=0$, $c(\tau)=P$, the corresponding form $S$ has a $\mathbb{Q}_w$-rational point if and only if there exists $x\in \mathbb{P}^1(\mathbb{Q}_w^t)$ such that $$\sigma(x)=x,\quad P. \tau(x)=x.$$ Suppose $E$ has a Weierstrass equation $y^2=F(x)$, where $F(x)\in \mathbb{Q}(x)$ is a monic irreducible cubic polynomial. Let $\alpha_1,\alpha_2,\alpha_3$ be the three roots of $F(x)$. We fix a embedding $\overline{\mathbb{Q}}\hookrightarrow \overline{\mathbb{Q}}_w$ and view $\alpha_i$ as elements in $\overline{\mathbb{Q}}_w$. Without loss of generality, we may assume that $\alpha_1\in \mathbb{Q}_w$ and thus $P=(\alpha_1,0)$. Then the action of $P$ on $\mathbb{P}^1$ is an involution that swaps $\alpha_1\leftrightarrow\infty$, $\alpha_2\leftrightarrow\alpha_3$. One can compute explicitly that this involution is given by the linear fractional transformation $$x\mapsto \frac{\alpha_1 x+(\alpha_2\alpha_3-\alpha_1\alpha_2-\alpha_1\alpha_3)}{x-\alpha_1}.$$ Therefore $Q_w(c)=0$ if and only if there exists $x\in \mathbb{P}^1(\mathbb{Q}_w^t)$ such that
  \begin{equation}
    \label{eq:fixedpoint}
    \sigma(x)=x,\quad (\tau(x)-\alpha_1)(x-\alpha_1)=(\alpha_1-\alpha_2)(\alpha_1-\alpha_3).
  \end{equation}

  Let $u$ be the prime of $L$ over $w$ induced by our fixed embedding $\overline{\mathbb{Q}}\hookrightarrow \overline{\mathbb{Q}}_w$. When $\Frob_w^2$ is sufficiently close to 1 (depending only on $E$), $u$ splits in the quadratic extension $L(\sqrt{\alpha_1-\alpha_2})/L$. Therefore $\alpha_1-\alpha_2\in (L_u^\times)^2$. The element $(\alpha_1-\alpha_2)(\alpha_1-\alpha_3)$, as the norm of $\alpha_1-\alpha_2$ from $L_u^\times$ to $\mathbb{Q}_w^\times$, must lie in $(\mathbb{Q}_w^\times)^2$. Let $\mathbb{Q}_w(\sqrt{\pi})$ be the tamely ramified quadratic extension fixed by $\sigma$. Then the image of the norm map $$\mathbb{N}: \mathbb{Q}_w(\sqrt{\pi})^\times\rightarrow \mathbb{Q}_w^\times,\quad y\mapsto y\cdot \tau(y)$$ has index two in $\mathbb{Q}_w^\times$ by local class field theory, and thus contains $(\mathbb{Q}_w^\times)^2$. So we can find $y\in \mathbb{Q}_w(\sqrt{\pi})^\times$ such that $\mathbb{N}(y)=(\alpha_1-\alpha_2)(\alpha_1-\alpha_3)$. Now $x=y+\alpha_1$ satisfies Equation (\ref{eq:fixedpoint}) and hence $Q_w(c)=0$. It follows that $\mathcal{L}_w'$ is an isotropic line for $Q_w$, as desired.
\end{proof}

\begin{lemma}\label{lem:raisingdensity}
Suppose $A$ is obtained from $E$ via level raising. Then there exists a positive density set of primes $w$ satisfying the following.
  \begin{enumerate}
  \item $H^1(\mathbb{Q}_w,V)$ is 2-dimensional.
  \item Let $\mathcal{L}_w'=\im(H^1(\mathbb{Q}_w, W)\rightarrow H^1(\mathbb{Q}_w, V))$, where $W$ is the unique $G_{\mathbb{Q}_w}$-stable line in $V$.  Then $\mathcal{L}_w'$ is an isotropic line for $Q_w$. 
  \item $\res_w(\Sel(A))=0$,
  \end{enumerate}
\end{lemma}

\begin{proof}

Observe that for primes $w$ such that $\Frob_w$ is sufficiently close to the class of the complex conjugation (depending only on $A$ and $E$), we have 
\begin{enumerate}
\item $\Frob_w\in \Gal(L/\mathbb{Q})$ has order 2 since we assumed $\Delta<0$ (Remark \ref{rem:negativedisc}). So $H^1(\mathbb{Q}_w,V)$ is 2-dimensional by Lemma \ref{lem:dimension}.
\item $\mathcal{L}_w'$ is an isotropic line for $Q_w$, by Lemma \ref{lem:isotropic}.
\item  Let $c_1,\ldots, c_r$ be a $k$-basis of $\Sel(A)$. Let $f_i: G_L\rightarrow V$ be the homomorphisms corresponding to $c_i$. Then $f_i(\Frob_w^2)=0$, hence $\res_w(c_i)=0$ for any $i\le r$. This is satisfied if $\Frob_w^2$ is trivial on the field cut out by the homomorphisms $f_1,\ldots,f_r$, which is a condition depending only on $A$.
\end{enumerate}
 The Chebotarev density theorem now finishes the proof.
\end{proof}

\begin{proposition}\label{pro:rankraising}
  Suppose $A$ is obtained from $E$ via level raising at primes $q_1,\ldots,q_m$ ($m\ge0$) satisfying that for any $i\le m$,
  \begin{enumerate}
  \item $H^1(\mathbb{Q}_{q_i},V)$ is 2-dimensional,
  \item $\varepsilon_{i}=\varepsilon_i(A)=+1$, and
  \item $\mathcal{L}_{q_i}=\mathcal{L}_{q_i}(A)$ is an isotropic line for $Q_{q_i}$.
  \end{enumerate}
  Then there exists a positive density set of primes $q_{m+1}$ and $A'$ obtained from $E$ via level raising at primes $q_1,\ldots,q_m, q_{m+1}$ satisfying that
  \begin{enumerate}
  \item   $H^1(\mathbb{Q}_{q_{m+1}},V)$ is 2-dimensional,
  \item  $\varepsilon_{i}'=\varepsilon_i(A')=+1$ for any $i\le m+1$,
  \item $\mathcal{L}_{q_i}'=\mathcal{L}_{q_i}(A')$ is an isotropic line for $Q_{q_i}$ for any $i\le m+1$, and
  \item  $\dim\Sel(A')=\dim\Sel(A)+1.$
  \end{enumerate}
\end{proposition}

\begin{proof}
There exists a positive density of primes $w=q_{m+1}$ satisfying the conditions in Lemma \ref{lem:raisingdensity}. For such $w$, we choose $A'$ using Theorem \ref{thm:levelraising} such that $\varepsilon_i'=+1$ for any $i\le m+1$. Let $\mathcal{L}=\mathcal{L}(A)$ and $\mathcal{L}'=\mathcal{L}(A')$. For $v=q_i$ ($i\le m$), we have $\dim H^1(\mathbb{Q}_v, V)=2$ and $\mathcal{L}_v=\mathcal{L}_v'$ is an isotropic line for $Q_v$ by the assumption and Lemma \ref{lem:localconditions}. Moreover, $\mathcal{L}_w$ and $\mathcal{L}_w'$ are distinct isotropic lines for $Q_w$ by Lemma \ref{lem:raisingdensity} and Lemma \ref{lem:localconditions}.  Then the conclusion (1-3) follows.  Notice $\mathcal{L}_v=\mathcal{L}_v'$ are maximal totally isotropic for $v\nmid q_1\cdots q_{m+1}$ by Lemma \ref{lem:localconditions} and Remark \ref{rem:poonenrains}. We can apply Lemma \ref{lem:rankraising} to obtain conclusion (4) .
\end{proof}

\begin{theorem}\label{thm:rankraising}
    Suppose $E$ satisfies Assumption \ref{ass:main}. Then for any given integer $n\ge \dim\Sel(E)$, there exists infinitely many abelian varieties $A$ obtained from $E$ via level raising, such that $$\dim\Sel(A)=n.$$
\end{theorem}

\begin{proof}
  The statement for $n>\dim\Sel(E)$ follows immediately from Proposition \ref{pro:rankraising} by induction on the number of level raising primes $m$. Applying Proposition \ref{pro:ranklowering} to $A$ with $\dim\Sel(A)=\dim\Sel(E)+1$ once, the statement for $n=\dim\Sel(E)$ also follows.
\end{proof}

Our main Theorem \ref{thm:mainarbitrary} then follows from Theorem \ref{thm:ranklowering} and \ref{thm:rankraising}.

\bibliographystyle{amsalpha} 
\bibliography{2-Selmer} 

\end{document}